\documentclass{article}%
\usepackage{amsmath,color}
\usepackage{amsfonts}
\usepackage{amssymb}
\usepackage{graphicx}
\usepackage{amsmath}%

\providecommand{\U}[1]{\protect\rule{.1in}{.1in}}

\newtheorem{theorem}{Theorem}

\renewcommand{\L}{\mathbb{L}}
\newcommand{\E}{\ensuremath{\mathbb{E}}}
\renewcommand{\P}{\ensuremath{\mathbb{P}}}
\newcommand{\R}{\ensuremath{\mathbb{R}}}
\newcommand{\hm}{h_{\min}}
\newcommand{\pen}{\mathrm{pen}}

\newcommand{\argmin}{\mathop{\mathrm{arg\,min}}}
\def\1{\mathds{1}}

\renewcommand{\H}{\mathcal H}

\newtheorem{corollary}[theorem]{Corollary}

\newtheorem{definition}[theorem]{Definition}

\newtheorem{lemma}[theorem]{Lemma}

\newtheorem{proposition}[theorem]{Proposition}

\newenvironment{proof}[1][Proof]{\noindent\textbf{#1.} }{\ \rule{0.5em}{0.5em}}
\begin{document}

\title{Estimator selection: a new method with applications to kernel density estimation}
\author{Claire Lacour
\footnote{Univ Paris-Sud, Laboratoire de Math\'ematiques d'Orsay, UMR8628, Orsay, F-91405. claire.lacour@u-psud.fr}
, Pascal Massart
\footnote{Univ Paris-Sud, Laboratoire de Math\'ematiques d'Orsay, UMR8628, Orsay, F-91405. pascal.massart@u-psud.fr}
, Vincent Rivoirard
\footnote{CEREMADE, CNRS, UMR 7534, Universit\'e Paris Dauphine, PSL Research University 75016 Paris, France.
rivoirard@ceremade.dauphine.fr}}

\maketitle

\begin{abstract}
Estimator selection has become a crucial issue in non parametric estimation.
Two widely used methods are penalized empirical risk minimization (such as
penalized log-likelihood estimation) or pairwise comparison (such as Lepski's
method). Our aim in this paper is twofold. First we explain some general ideas
about the calibration issue of estimator selection methods. We review some
known results, putting the emphasis on the concept of minimal penalty which is
helpful to design data-driven selection criteria. Secondly we present a new
method for bandwidth selection within the framework of kernel density density
estimation which is in some sense intermediate between these two main methods
mentioned above. We provide some theoretical results which lead to some fully
data-driven selection strategy.\\


\end{abstract}

\section{Introduction}

Since the beginning of the 80's and the pioneering works of Pinsker (see
\cite{pinsker} and \cite{efro-pinsk}) many efforts have been made in
nonparametric statistics to design and analyze adaptive estimators. Several
breakthroughs have been performed in the 90's. Lepski initiated a strategy
mainly based on a new bandwidth selection method for kernel estimators (see
\cite{lep1} and \cite{lep2}) while Donoho, Johnstone, Kerkyacharian and Picard
devoted a remarkable series of works to adaptive properties of wavelet
thresholding methods (see \cite{DoJo}, \cite{dojo2} and \cite{djkpasymp} for
an overview). In the same time, following a somehow more abstract
information-theoretic path Barron and Cover \cite{bar-cov} have built adaptive
density estimators by using some minimum complexity principle on discrete
models, paving the way for a more general connection between model selection
and adaptive estimation that has been performed in a series of works by
Birg\'{e} and Massart (see \cite{BM97a}, \cite{BBM} and \cite{mas-stflour} for
an overview). Estimator selection is a common denominator between all these
works. In other words, if one observes some random variable $\xi$ (which can
be a random vector or a random process) with unknown distribution, and one
wants to estimate some quantity $f$ (the target) related to the distribution
of $\xi$, a flexible approach to estimate some target $f$ is to consider some
collection of preliminar estimators $\left\{  \hat{f}_{m},m\in\mathcal{M}%
\right\}  $ and then try to design some genuine data-driven procedure $\hat
{m}$ to produce some new estimator $\hat{f}_{\hat{m}}$. We have more precisely
in mind frameworks in which the observation $\xi$ depends on some parameter
$n$ (the case $\xi$ $=\left(  \xi_{1},...,\xi_{n}\right)  $, where the
variables $\xi_{1},...,\xi_{n}$ are independent and identically distributed is
a typical example) and when we shall refer to asymptotic results, this will
implicitly mean that $n$ goes to infinity. The model selection framework
corresponds to the situation where each estimator $\hat{f}_{m}$ is linked to
some model $S_{m}$ through some empirical risk minimization procedure (least
squares for instance) but bandwidth selection for kernel estimators also falls
into this estimator selection framework. Since the beginning of the century,
the vigorous developments of high-dimensional data analysis led to the
necessity of building new inference methods and accordingly new mathematical
analysis of the performance of these methods. Dealing with high-dimensional
data also raises new questions related to the implementation of the methods:
you do not only need to design estimation methods that performs well from a
statistical point of view but you also need fast algorithms. In those
situations the estimator selection issue is still relevant but you may
encounter situations for which not only the estimators but the collection
$\mathcal{M}$ itself depends on the data. The celebrated Lasso algorithm
introduced by Tibshirani in \cite{Tib} (and mathematically studied in
\cite{bic-tsy}) perfectly illustrates this fact, since in this case the
regularization path of the Lasso naturally leads to some data-dependent
ordered collection $\mathcal{M}$ of set of variables. Estimator selection has
therefore become a central issue in nonparametric statistics. To go further
into the subject one needs to consider some loss function $\ell$ allowing to
measure the quality of each estimator $\hat{f}_{m}$ by $\ell\left(  f,\hat
{f}_{m}\right)  $. If the target to be estimated as well as the estimators
belong to some Hilbert space $\left(  \mathbb{H},\left\Vert .\right\Vert
\right)  $, the quadratic loss $\ell\left(  f,\hat{f}_{m}\right)  =\left\Vert
f-\hat{f}_{m}\right\Vert ^{2}$ is a standard choice, but of course other
losses are considered in the literature (typically powers of $\mathbb{L}_{p}$
norms when $f$ is a function). Given this loss function, the quality of a
selection procedure $\hat{m}$ is then measured by $\ell\left(  f,\hat{f}%
_{\hat{m}}\right)  $ and mathematical results on estimator selection are
formulated in terms of upper bounds (either in probability or in expectation)
on $\ell\left(  f,\hat{f}_{\hat{m}}\right)  $ that allow to measure how far
this quantity is from what is usually called the \textit{oracle} risk
$\inf_{m\in\mathcal{M}}\mathbb{E}\ell\left(  f,\hat{f}_{m}\right)  $. These
comparison inequalities are called oracle inequalities. A common feature of
estimator selection methods such as penalized empirical risk minimization or
Lepski's method (or more recently Glodenshluger-Lepski's method) is that they
involve some "hyper-parameter" $\lambda$ say that needs to be chosen by the
statistician (typically $\lambda$ is some multiplicative constant in a penalty
term). Positive mathematical results on a given selection procedure typically
tells you that if $\lambda$ remains above some quantity $\lambda_{\min}$ then
you are able to prove some oracle inequality. Now the point is that in
practice these methods are very sensitive to this choice $\lambda$. In
particular, too small value of $\lambda$ should be prohibited since they lead
to overfitting. Birg\'{e} and Massart have introduced in \cite{minpen} the
concept of \textit{minimal penalty}. They prove that in the Gaussian white
noise framework, there is number of model selection problems for which it is
possible to prove the existence of a critical value $\lambda_{\min}$ for
$\lambda$ (which precisely corresponds to the minimal penalty) below which the
selection procedure dramatically fails in the sense that the risk of the
selected estimator is of order $\sup_{m\in\mathcal{M}}\mathbb{E}\ell\left(
f,\hat{f}_{m}\right)  $ instead of $\inf_{m\in\mathcal{M}}\mathbb{E}%
\ell\left(  f,\hat{f}_{m}\right)  $! Moreover they also prove some oracle
inequalities showing that taking $\lambda$ as twice this critical value
$\lambda_{\min}$ corresponds to the "best" choice for $\lambda$ (at least
asymptotically i.e. when the level of noise goes to zero). Since this critical
value may be detected from the data (see the slope heuristics described
below), the estimated critical value $\hat{\lambda}_{\min}$ leads to an
entirely data-driven choice $\hat{\lambda}=2\hat{\lambda}_{\min}$ for the
hyper parameter $\lambda$ which can be proved to be (nearly) optimal for some
model selection problems. This approach for calibrating the penalty has been
implemented, tested on simulation studies and used on real in several papers
devoted to model selection (see \cite{lebar} for the first simulation studies
and \cite{bau} for an overview on practical aspects). It has also been
mathematically studied in several statistical frameworks that differs from the
Gaussian white noise (see \cite{arlot-mass}, \cite{saumard1}, \cite{lerasle}
and \cite{lerasle-tak} for instance). In there very nice paper
\cite{arlot-bach}, Arlot and Bach have adapted this strategy (based on the
concept of minimal penalty) for calibrating a least squares penalized
criterion to the context of selection of (linear) estimators which are not
necessarily least squares estimators. In the same spirit, Lerasle,
Malter-Magalahes and Reynaud-Bouret have computed minimal and optimal penalty
formulas for the problem of kernel density estimators selection via some
penalized least square criterion (see \cite{lerasle-nelo}). Minimal penalties for Lasso type estimates can also be also revealed as illustrated by \cite{bertin11:_adapt_dantiz} in the framework of density estimation.
Very recently,
Lacour and Massart have shown in \cite{lac-mass} that the concept of minimal
penalty also makes sense for the Goldenshluger-Lepski method in the context of
kernel estimators bandwidth selection. Unlike penalized least squares (or
other risk minimization criteria), the Goldenshluger-Lepski method is a
pairwise comparison based estimator selection procedure and the minimal
penalty result given in \cite{lac-mass} is a first step in the direction of
designing entirely data-driven calibration strategies for such pairwise
comparison methods. From a mathematical view point, it is quite well known now
that the proofs of oracle inequalities heavily rely on right tail deviation
inequalities for stochastic quantities such as suprema of empirical (or
Gaussian) processes. These deviation inequalities may derive or not from
concentration inequalities and as a matter of fact many oracle inequalities
have been obtained by using deviation inequalities which are not concentration
inequalities per se (see \cite{lep3} for instance). The price to pay is that
the resulting oracle inequalities typically require that $\lambda>\lambda_{0}$
for some value $\lambda_{0}$ which can be much larger than $\lambda_{\min}$.
Interestingly, concentration inequalities lead to sharper results that allows
(at least for some specific situations for which things are computable) to
derive the computation of $\lambda_{\min}$. Oracle inequalities are obtained
by using right tail concentration while the minimal penalty results are proved
by using left tails concentration inequalities. Our purpose in this paper is
first to recall some heuristics on the penalty calibration method mentioned
above in the context of penalized model selection. Then we briefly explain
(following the lines of \cite{minpen}) why concentration is the adequate tool
to prove a minimal penalty result. Finally we introduce some new estimator
selection method which lies somewhere between penalized empirical risk
minimization and the pairwise comparison methods (such as the
Goldenshluger-Lepski's method). We study this new method in the context of
bandwidth selection for kernel density estimators with the squared
$\mathbb{L}_{2}$ loss. We provide oracle inequalities and minimal penalty
results that allow to compute minimal and optimal values for the penalty.
\section{Penalized model selection}\label{sec:Penalized}
Viewing model selection as a special instance of estimator selection requires
to formulate the model selection problem as an estimation problem of some
target quantity $f$ (typically, $f$ is a function) belonging to some space
$\mathcal{S}$ say as suggested above. In the model selection framework, the
list of estimators and the loss function are intimately related in the sense
that they derive from the same \textit{contrast function} (also called
\textit{empirical risk} in the machine learning literature). More precisely,
a contrast function $L_{n}$ is a function on the set $\mathcal{S}$ depending
on the observation $\xi^{\left(  n\right)  }$ in such a way that
\[
g\rightarrow\mathbb{E}\left[  L_{n}\left(  g\right)  \right]
\]
achieves a minimum at point $f$. Given some collection of subsets $\left(
S_{m}\right)  _{m\in\mathcal{M}}$ of $\mathcal{S}$ (these subsets are simply
called models in what follows), for every $m\in\mathcal{M}$, some estimator
$\hat{f}_{m}$ of $f$ is obtained by minimizing $L_{n}$ over $S_{m}$ ($\hat
{f}_{m}$ is called \textit{minimum contrast estimator} or \textit{empirical
risk minimizer}). On the other hand some \textquotedblleft
natural\textquotedblright\ (non negative) loss function can also be attached to
$L_{n}$ through the simple definition
\begin{equation}
\ell\left(  f,g\right)  =\mathbb{E}\left[  L_{n}\left(  g\right)  \right]
-\mathbb{E}\left[  L_{n}\left(  f\right)  \right]  \text{ \ } \label{e1l}%
\end{equation}
for all \ $g\in\mathcal{S}$. In the case where $\xi^{\left(  n\right)
}=\left(  \xi_{1},...,\xi_{n}\right)  $, an empirical criterion $L_{n}$ can be
defined as an empirical mean
\[
L_{n}\left(  g\right)  =P_{n}\left[  L\left(  g,.\right)  \right]  :=\frac
{1}{n}\sum_{i=1}^{n}L\left(  g,\xi_{i}\right)  \text{,}%
\]
which justifies the terminology of empirical risk. Empirical risk minimization
includes maximum likelihood and least squares estimation as briefly recalled
below in the contexts of density estimation and Gaussian white noise.

\begin{itemize}
\item \textbf{Density estimation}
\end{itemize}

One observes $\xi_{1},...,\xi_{n}$ which are i.i.d. random variables with
unknown density $f$ with respect to some given measure $\mu$. The choice
\[
L\left(  g,x\right)  =-\ln\left(  g\left(  x\right)  \right)
\]
leads to maximum likelihood estimation and the corresponding loss function
$\ell$ is given by
\[
\ell\left(  f,g\right)  =K\left(  f,g\right)  \text{,}%
\]
where $K\left(  f,g\right)  $ denotes the Kullback-Leibler divergence between
the probabilities $f\mu$ and $g\mu$, i.e.
\[
K\left(  f,g\right)  =\int f\ln\left(  \frac{f}{g}\right)  d\mu
\]
if $f\mu$ is absolutely continuous with respect to $g\mu$ and $K\left(
f,g\right)  =+\infty$ otherwise. Assuming that $f\in\mathbb{L}_{2}\left(
\mu\right)  $, it is also possible to define a least squares density
estimation procedure by setting this time
\[
L\left(  g,x\right)  =\left\Vert g\right\Vert ^{2}-2g\left(  x\right)
\]
where $\left\Vert .\right\Vert $ denotes the norm in $\mathbb{L}_{2}\left(
\mu\right)  $ and the corresponding loss function $\ell$ is in this case given
by
\[
\ell\left(  f,g\right)  =\left\Vert f-g\right\Vert ^{2}\text{,}%
\]
for every $g\in\mathbb{L}_{2}\left(  \mu\right)  $.

\begin{itemize}
\item \textbf{Gaussian white noise}
\end{itemize}

Let us consider the general Gaussian white noise framework as introduced in
\cite{BM99g}. This means that, given some separable Hilbert space%
\index{Hilbert space}
$\mathbb{H}$, one observes
\begin{equation}
\xi\left(  g\right)  =\left\langle f,g\right\rangle +\varepsilon W\left(
g\right)  \text{, for all }g\in\mathbb{H}\text{,} \label{emodeliso}%
\end{equation}
where $W$ is some isonormal process, i.e. $W$ maps isometrically $\mathbb{H}$
onto some Gaussian subspace of $\mathbb{L}_{2}\left(  \Omega\right)  $. This
framework is convenient to cover both the infinite dimensional white noise
model for which $\mathbb{H}=\mathbb{L}_{2}\left(  \left[  0,1\right]  \right)
$ and $W\left(  g\right)  =\int_{0}^{1}g\left(  x\right)  dB\left(  x\right)
$, where $B$ is a standard Brownian motion, and the finite dimensional linear
model for which $\mathbb{H}=\mathbb{R}^{n}$ and $W\left(  g\right)
=\left\langle \eta,g\right\rangle =\frac{1}{n}\sum_{i=1}^{n}g_{i}\eta_{i}$,
where $\eta$ is a standard $n$-dimensional Gaussian vector. In what follows,
we shall write the level of noise $\varepsilon$ as $\varepsilon=1/\sqrt{n}$.
The introduction of $n$ here is purely artificial, it just helps in keeping
homogenous notations when passing from the density estimation framework to the
Gaussian white noise framework (note that this choice is not misleading since
in the $n$-dimensional case mentioned above the level of noise is actually
equal to $1/\sqrt{n}$). The least squares criterion is defined by
\[
L_{n}\left(  g\right)  =\left\Vert g\right\Vert ^{2}-2\xi\left(  g\right)
\text{,}%
\]
and the corresponding loss function $\ell$ is simply the quadratic loss,
\[
\ell\left(  f,g\right)  =\left\Vert f-g\right\Vert ^{2}\text{,}%
\]
for every $g\in\mathbb{H}$.

\subsection{The model choice paradigm}

The choice of a model $S_{m}$ on which the empirical risk minimizer is to be
defined is a challenging issue. It is generally difficult to guess what is the
right model to consider in order to reflect the nature of data from the real
life and one can get into problems whenever the model $S_{m}$ is false in the
sense that the true $f$ is too far from $S_{m}$. One could then be tempted to
choose $S_{m}$ as big as possible. Taking $S_{m}$ as $\mathcal{S}$ itself or
as a huge\ subset of $\mathcal{S}$ is known to lead to inconsistent (see
\cite{Baha}) or suboptimal estimators (see \cite{BM93}). Choosing some model
$S$ in advance leads to some difficulties

\begin{itemize}
\item If $S_{m}$ is a \textit{small}\ model (think of some parametric model,
defined by $1$ or $2$ parameters for instance) the behavior of a minimum
contrast estimator on $S_{m}$ is satisfactory as long as $f$ is close enough
to $S_{m}$ but the model can easily turn to be false.

\item On the contrary, if $S_{m}$ is a \textit{huge}\ model (think of the set
of all continuous functions on $\left[  0,1\right]  $ in the regression
framework for instance), the minimization of the empirical criterion leads to
a very poor estimator of $f$ even if $f$ truly belongs to $S_{m}$.
\end{itemize}

This is by essence what makes the data-driven choice of a model an attractive
issue. Interestingly the risk of the empirical risk minimizer on a given model
$S_{m}$ is a meaningful criterion to measure the quality of the model as
illustrated by the following example.

\textbf{Illustration (white noise)}

In the white noise framework, if one takes $S_{m}$ as a linear subspace with
dimension $D_{m}$ of $\mathbb{H}$, one can compute the least squares estimator
explicitly. Indeed, if $\left(  \phi_{j}\right)  _{1\leq j\leq D_{m}}$ denotes
some orthonormal basis of $S$, the corresponding least squares estimator is
merely a projection estimator
\[
\hat{f}_{m}=\sum_{j=1}^{D_{m}}\xi\left(  \phi_{j}\right)  \phi_{j}\text{.}%
\]
Since for every $1\leq j\leq D_{m}$%
\[
\xi\left(  \phi_{j}\right)  =\left\langle f,\phi_{j}\right\rangle +\frac
{1}{\sqrt{n}}\eta_{j}%
\]
where the variables $\eta_{1},...,\eta_{D}$ are i.i.d. standard normal
variables, the expected quadratic risk of $\hat{f}_{m}$ can be easily
computed
\[
\mathbb{E}\left[  \left\Vert f-\hat{f}_{m}\right\Vert ^{2}\right]
=d^{2}\left(  f,S_{m}\right)  +\frac{D_{m}}{n}\text{.}%
\]
This formula for the quadratic risk perfectly reflects the model choice
paradigm since if one wants to choose a model in such a way that the risk of
the resulting least square estimator is small, we have to warrant that the
bias term $d^{2}\left(  f,S_{m}\right)  $ and the variance term $D_{m}/n$ are
small simultaneously. Therefore if $\left(  S_{m}\right)  _{m\in\mathcal{M}}$
is a list of finite dimensional subspaces of $\mathbb{H}$ and $\left(  \hat
{f}_{m}\right)  _{m\in\mathcal{M}}$ be the corresponding list of least square
estimators, it is relevant to consider that an \textit{ideal}\ model should
minimize $\mathbb{E}\left[  \left\Vert f-\hat{f}_{m}\right\Vert ^{2}\right]  $
with respect to $m\in\mathcal{M}$, which amounts to consider this model
selection problem as an instance of estimator selection.

More generally if we consider some empirical contrast $L_{n}$ and some
collection of models $\left(  S_{m}\right)  _{m\in\mathcal{M}}$, each model
$S_{m}$ is represented by the corresponding empirical risk minimizer $\hat
{f}_{m}$ related to $L_{n}$. The criterion $L_{n}$ also provides a natural
loss function $\ell$ and the benchmark for this estimator selection problem is
merely $\inf_{m\in\mathcal{M}}$ $\mathbb{E}\left[  \ell\left(  f,\widehat
{f}_{m}\right)  \right]  $. In other words, one would ideally like to select
$m\left(  f\right)  $ minimizing the risk $\mathbb{E}\left[  \ell\left(
f,\widehat{f}_{m}\right)  \right]  $ with respect to $m\in\mathcal{M}$. Note
that the resulting \textit{oracle} model $S_{m\left(  f\right)  }$ is not
necessarily a \textit{true }model and in this sense this approach to model
selection differs from other approaches which are based on the concept of a
true model. Nevertheless we stick now to the estimator selection approach
presented above and the issue becomes to design data-driven criteria to select
an estimator which tends to mimic an oracle, i.e. one would like the risk of
the selected estimator $\hat{f}_{\hat{m}}$ to be as close as possible to the
benchmark $\inf_{m\in\mathcal{M}}$ $\mathbb{E}\left[  \ell\left(
f,\widehat{f}_{m}\right)  \right]  $.

\subsection{Model selection via penalization}

The penalized empirical risk selection procedure consists in considering some
proper penalty function $\operatorname*{pen}$: $\mathcal{M}\rightarrow
\mathbb{R}_{+}$ and take $\hat{m}$ minimizing
\[
L_{n}\left(  \hat{f}_{m}\right)  +\operatorname*{pen}\left(  m\right)
\]
over $\mathcal{M}$. We can then define the selected model $S_{\hat{m}}$ and
the corresponding selected estimator $\hat{f}_{\hat{m}}$.

Penalized criteria have been proposed in the early seventies by Akaike or
Schwarz (see \cite{Akaike} and \cite{schwartz}) for penalized maximum
log-likelihood in the density estimation framework and Mallows for penalized
least squares regression (see \cite{dan-wood} and \cite{Mallows}). In both
cases the penalty functions are proportional to the number of parameters
$D_{m}$ of the corresponding model $S_{m}$

\begin{itemize}
\item Akaike (AIC): $D_{m}/n$

\item Schwarz (BIC): $\ln\left(  n\right)  D_{m}/\left(  2n\right)  $

\item Mallows ($C_{p}$): $2D_{m}\sigma^{2}/n$,
\end{itemize}

where the variance $\sigma^{2}$ of the errors of the regression framework is
assumed to be known by the sake of simplicity.

Akaike's or Schwarz's proposals heavily rely on the assumption that the
dimensions and the number of the models are bounded w.r.t. $n$ and $n$ tends
to infinity.

Viewing model selection as an instance of estimator selection is much more
related to the non asymptotic view of model selection. In other words, since
the purpose of the game is to mimic the oracle, it makes sense to consider
that $n$ is what it is and to allow the list of models as well as the models
themselves to depend on $n$.

As we shall see, concentration inequalities are deeply involved both in the
construction of the penalized criteria and in the study of the performance of
the resulting penalized estimator $\widehat{f}_{\widehat{m}}$ but before going
to the mathematical heart of this paper, it is interesting to  revisit
Akaike's heuristics in order to address the main issue: how to calibrate the
penalty in order to mimic the oracle?

\subsection{Revisiting Akaike's heuristics}

If one wants to better understand how to penalize optimally and the role that
concentration inequalities could play in this matter, it is very instructive
to come back to the root of the topic of model selection via penalization i.e.
to Mallows' and Akaike's heuristics which are both based on the idea of
estimating the risk in an unbiased way (at least asymptotically as far as
Akaike's heuristics is concerned). The idea is the following.

Let us consider, in each model $S_{m}$ some minimizer $f_{m}$ of
$t\rightarrow\mathbb{E}\left[  L_{n}\left(  t\right)  \right]  $ over $S_{m}$
(assuming that such a point does exist). Defining for every $m\in\mathcal{M}$,%
\[
\widehat{b}_{m}=L_{n}\left(  f_{m}\right)  -L_{n}\left(  f\right)  \text{ and
}\widehat{v}_{m}=L_{n}\left(  f_{m}\right)  -L_{n}\left(  \widehat{f}%
_{m}\right)  \text{,}%
\]
minimizing some penalized criterion%
\[
L_{n}\left(  \widehat{f}_{m}\right)  +\operatorname{pen}\left(  m\right)
\]
over $\mathcal{M}$ amounts to minimize%
\[
\widehat{b}_{m}-\widehat{v}_{m}+\operatorname{pen}\left(  m\right)  \text{.}%
\]
The point is that  $\widehat{b}_{m}$ is an unbiased estimator of the bias
term $\ell\left(  f,f_{m}\right)  $. If we have in mind to use concentration
arguments, one can hope that minimizing the quantity above will be
approximately equivalent to minimize
\[
\ell\left(  f,f_{m}\right)  -\mathbb{E}\left[  \widehat{v}_{m}\right]
+\operatorname{pen}\left(  m\right)  \text{.}%
\]
Since the purpose of the game is to minimize the risk $\mathbb{E}\left[
\ell\left(  f,\widehat{f}_{m}\right)  \right]  $, an ideal penalty would
therefore be
\[
\operatorname{pen}\left(  m\right)  =\mathbb{E}\left[  \widehat{v}_{m}\right]
+\mathbb{E}\left[  \ell\left(  f_{m},\widehat{f}_{m}\right)  \right]  \text{.}%
\]
In the Mallows' $C_{p}$ case, the models $S_{m}$ are linear and $\mathbb{E}%
\left[  \widehat{v}_{m}\right]  =\mathbb{E}\left[  \ell\left(  f_{m}%
,\widehat{f}_{m}\right)  \right]  $ are explicitly computable (at least if the
level of noise is assumed to be known). For Akaike's penalized log-likelihood
criterion, this is similar, at least asymptotically. More precisely, one uses
the fact that
\[
\mathbb{E}\left[  \widehat{v}_{m}\right]  \approx\mathbb{E}\left[  \ell\left(
f_{m},\widehat{f}_{m}\right)  \right]  \approx D_{m}/\left(  2n\right)
\text{,}%
\]
where $D_{m}$ stands for the number of parameters defining model $S_{m}$.

If one wants to design a fully data-driven penalization strategy escaping to
asymptotic expansions, a key notion introduced in \cite{minpen} is the concept
of minimal penalty. Intuitively if the penalty is below the critical quantity
$\mathbb{E}\left[  \widehat{v}_{m}\right]  $, then the penalized model
selection criterion should fail in the sense that it systematically selects a
model with high dimension (the criterion explodes). Interestingly, this
typical behavior helps in estimating the minimal penalty from the data. For
instance, if one takes a penalty of the form $\operatorname{pen}\left(
m\right)  =\lambda D_{m}$ one can hope to estimate the minimal value for
$\lambda$ from the data by taking $\hat{\lambda}_{\min}$ as the greatest value
$\lambda$ for which the penalized criterion with penalty $\lambda D_{m}$ does
explode. A complementary idea which is also introduced in \cite{minpen} is
that you may expect the optimal penalty to be related to the minimal penalty.
The relationship which is investigated in \cite{minpen} is especially simple,
since it simply amounts to take the optimal penalty as twice the minimal
penalty. Heuristically it is based on the guess that the approximation
$\mathbb{E}\left[  \widehat{v}_{m}\right]  \approx\mathbb{E}\left[
\ell\left(  f_{m},\widehat{f}_{m}\right)  \right]  $ remains valid even in a
non asymptotic perspective. If this belief is true then the minimal penalty is
$\mathbb{E}\left[  \widehat{v}_{m}\right]  $ while the optimal penalty should
be $\mathbb{E}\left[  \widehat{v}_{m}\right]  +\mathbb{E}\left[  \ell\left(
f_{m},\widehat{f}_{m}\right)  \right]  $ and their ratio is approximately
equal to $2$. Hence $2\mathbb{E}\left[  \widehat{v}_{m}\right]  $ (which is
exactly twice the minimal penalty $\mathbb{E}\left[  \widehat{v}_{m}\right]
$!) turns out to be a good approximation of the ideal penalty. Coming back to
the typical case when the penalty has the form $\operatorname{pen}\left(
m\right)  =\lambda D_{m}$ this leads to an easy to implement rule of thumb to
finally choose the penalty from the data as $\operatorname{pen}\left(
m\right)  =2\hat{\lambda}_{\min}D_{m}$. In some sense explains the rule of
thumb which is given in the preceding Section: the minimal penalty is
$\widehat{v}_{m}$ while the optimal penalty should be $\widehat{v}%
_{m}+\mathbb{E}\left[  \ell\left(  f_{m},\widehat{f}_{m}\right)  \right]  $
and their ratio is approximately equal to $2$.

Implicitly we have made as if the empirical losses $\widehat{v}_{m}$ were
close to their expectations for all models at the same time. The role of
concentration arguments will be to validate this fact, at least if the list of
models is not too rich. In practice this means that starting from a given list
of models, one has first to decide to penalize in the same way the models
which are defined by the same number of parameters. Then one considers a new
list of models $\left(  S_{D}\right)  _{D\geq1}$, where for each integer $D$,
$S_{D}$ is the union of those among the initial models which are defined by
$D$ parameters and then apply the preceding heuristics to this new list.

\subsection{Concentration inequalities in action: a case example}\label{sec:minimal}

Our aim in this section is to explain the role of concentration inequalities
in the derivation of minimal penalty results by studying the simple but
illustrative example of ordered variable selection within the Gaussian
framework. More precisely, let us consider some infinite dimensional separable
Hilbert space $\mathbb{H}$ and an orthonormal basis $\left\{  \phi_{j}%
,j\in\mathbb{N}^{\ast}\right\}  $ of $\mathbb{H}$. Let us assume that one
observes the Gaussian white noise process $\xi$ defined by \eqref{emodeliso}.
Given some arbitrary integer $N$, the ordered variable selection problem
consists in selecting a proper model $S_{\hat{D}}$ among the collection
$\left\{  S_{D},1\leq D\leq N\right\}  $, where for every $D$, $S_{D}$ is
defined as the linear span of $\left\{  \phi_{j},j\leq D\right\}  $. The
penalized least squares procedure consists of selecting $\hat{D}$ minimizing
over $D\in\left[  1,N\right]  $, the criterion $\left\Vert \hat{f}%
_{D}\right\Vert ^{2}-2\xi\left(  \hat{f}_{D}\right)  +\operatorname*{pen}%
\left(  D\right)  $, where $\hat{f}_{D}$ is merely the projection estimator
$\hat{f}_{D}=\sum_{j=1}^{D}\xi\left(  \phi_{j}\right)  \phi_{j}$. Equivalently
$\hat{D}$ minimizes
\[
\operatorname*{crit}\left(  D\right)  =-%
{\displaystyle\sum\limits_{j=1}^{D}}
\xi^{2}\left(  \phi_{j}\right)  +\operatorname*{pen}\left(  D\right)
\]
over $D\in\left[  1,N\right]  $. We would like to show why the value
$\lambda=1$ is critical when the penalty is defined as $\operatorname*{pen}%
\left(  D\right)  =\lambda D\varepsilon^{2}$. To make the problem even simpler
let us also assume that $f=0$ (this assumption can be relaxed of course and
the interested reader will find in \cite{minpen} a complete proof of the
minimal penalty result under much more realistic assumptions!). In this case
the empirical loss can merely be written as%
\[%
{\displaystyle\sum\limits_{j=1}^{D}}
\xi^{2}\left(  \phi_{j}\right)  =\varepsilon^{2}%
{\displaystyle\sum\limits_{j=1}^{D}}
W^{2}\left(  \phi_{j}\right)  =\varepsilon^{2}\chi_{D}^{2}%
\]
where the variable $\chi_{D}^{2}$ follows a chi-squared distribution with $D$
degrees of freedom. It is easy now to see how fundamental concentration
inequalities are for proving a minimal penalty result. Due to the Gaussian
framework that we use here, it's not surprise that the celebrated Gaussian
concentration of measure phenomenon can be used. More precisely, we may apply
the Gaussian concentration inequality and the Gaussian Poincar\'{e} inequality
(see \cite{boucheron-2013} for instance) which ensure that on the one hand
defining $Z_{D}$ as either $\chi\left(  D\right)  -\mathbb{E}\left[
\chi\left(  D\right)  \right]  $ or $-\chi\left(  D\right)  +\mathbb{E}\left[
\chi\left(  D\right)  \right]  $,
\[
\mathbb{P}\left\{  Z_{D}\geq\sqrt{2x}\right\}  \mathbb{\leq}e^{-x}\text{, for
all positive }x
\]
and on the other hand%
\[
0\leq\mathbb{E}\left[  \chi_{D}^{2}\right]  -\left(  \mathbb{E}\left[
\chi_{D}\right]  \right)  ^{2}\leq1\text{.}%
\]
Since of course $\mathbb{E}\left[  \chi_D^{2}\right]  =D$, this leads to the
right tail inequality
\begin{equation}
\mathbb{P}\left\{  \chi_{D}-\sqrt{D}\geq\sqrt{2x}\right\}  \mathbb{\leq}e^{-x}
\label{echidroit}%
\end{equation}
and to the left tail inequality%
\begin{equation}
\mathbb{P}\left\{  \chi_{D}-\sqrt{D-1}\geq\sqrt{2x}\right\}  \mathbb{\leq
}e^{-x}\text{.} \label{echigauche}%
\end{equation}
Taking the penalty as $\operatorname*{pen}\left(  D\right)  =\lambda
D\varepsilon^{2}$, here is what can be proved by using these inequalities (The
result below is absolutely not new. It is a consequence of much more general
results proved in \cite{minpen}. We recall it with its proof for expository reasons).

\begin{proposition}
\label{bou1} Given $\delta$ and $\lambda$ in $(0,1)$, if $N$ is large enough
depending only on $\lambda$ and $\delta$, and if the penalty is taken as
\begin{equation}
\operatorname*{pen}\left(  D\right)  =\lambda D\varepsilon^{2}\text{, for all
}D\in\left[  1,N\right]  \label{pensub}%
\end{equation}
then
\[
\mathbb{P}\left\{  \hat{D}\geq\frac{\left(  1-\lambda\right)  }{2}N\right\}
\geq1-\delta\text{ and }\mathbb{E}\left[  \left\Vert \widehat{f}_{\widehat{m}%
}\right\Vert ^{2}\right]  \mathbb{\geq}\frac{(1-\delta)(1-\lambda)}%
{4}N\varepsilon^{2}.
\]

\end{proposition}

%

\proof
\vspace{2mm}\newline

In order to compare the values of the penalized criterion at points $D$ and
$N$, we write
\begin{equation}
\varepsilon^{-2}[\operatorname*{crit}\left(  D\right)  -\operatorname*{crit}%
\left(  N\right)  ]\geq\chi_{N}^{2}-\chi_{D}^{2}-\lambda N. \label{ecrit}%
\end{equation}
Now, recalling that $\lambda<1$ we set
\begin{equation}
0<\eta=\left(  1-\lambda\right)  /2<1/2\text{; }\mathcal{\theta}=\eta
^{2}/48\text{.} \label{eaux}%
\end{equation}
Assume that $N$ is large enough (depending on $\delta$ and $\lambda$) to
ensure that the following inequalities hold:
\begin{equation}
e^{-\mathcal{\theta}N\eta}\sum_{D\geq1}e^{-\mathcal{\theta}D}\leq\delta
;\qquad\mathcal{\theta}N\eta\geq1/6\text{.} \label{eaux1}%
\end{equation}
Let us introduce the event
\begin{align*}
\bar{\Omega}  &  =\left[  \bigcap_{D<N\eta}\left\{  \chi_{D}\leq\sqrt{D}%
+\sqrt{2\mathcal{\theta}\left(  D+N\eta\right)  }\right\}  \right] \\
&  \bigcap\;\left[  \bigcap_{{N\eta\leq D}}\left\{  \chi_{D}\geq\sqrt{D-1}%
-\sqrt{2\mathcal{\theta}\left(  D+N\eta\right)  }\right\}  \right]  .
\end{align*}
Using either (\ref{echidroit}) if $D<N\eta$ or (\ref{echigauche}) if $D\geq
N\eta$ , we get by (\ref{eaux1})
\[
\mathbb{P}\left(  \bar{\Omega}^{c}\right)  \leq\sum_{D\geq1}%
e^{-\mathcal{\theta}(D+N\eta)}\leq\delta.
\]
Moreover, on $\overline{\Omega}$, $\chi_{D}^{2}\leq\left(  1+2\sqrt
{\mathcal{\theta}}\right)  ^{2}N\eta$, for all $D$ such that $D<N\eta$ and, by
(\ref{eaux}) and (\ref{eaux1}), $\chi_{N}\geq\sqrt{N-1}-\sqrt{3\mathcal{\theta
}N}$ and $\mathcal{\theta}N>1/3$. Therefore $\chi_{N}^{2}\geq N\left(
1-2\sqrt{3\mathcal{\theta}}\right)  $. Hence, on $\bar{\Omega}$, (\ref{ecrit})
and (\ref{eaux}) yield
\begin{align*}
\varepsilon^{-2}[\operatorname{crit}\left(  D\right)  -\operatorname{crit}%
\left(  N\right)  ]  &  \geq N\left(  1-2\sqrt{3\mathcal{\theta}}\right)
-\left(  1+2\sqrt{\mathcal{\theta}}\right)  ^{2}N\eta-\lambda N\\
&  >(1-\frac{\eta}{2})N-\frac{3}{2}\eta N-(1-2\eta)N\;\;=\;\;0,
\end{align*}
for all $D$ such that $D<N\eta$. This immediately implies that $\hat{D}$
cannot be smaller than $N\eta$ on $\overline{\Omega}$ and therefore,
\begin{equation}
\mathbb{P}\left\{  \hat{D}\geq N\eta\right\}  \geq\mathbb{P}\left(
\bar{\Omega}\right)  \geq1-\delta. \label{eminpro}%
\end{equation}
Moreover, on the same set $\bar{\Omega}$, $\chi_{D}\geq\sqrt{D-1}%
-\sqrt{2\mathcal{\theta}\left(  D+N\eta\right)  }$ if $D$ is such that $D\geq
N\eta$. Noticing that $N\eta>32$ and recalling that $\eta\leq1/2$, we derive
that on the set $\overline{\Omega}$ if $D$ is such that $D\geq N\eta$%
\[
\chi_{D}\geq\sqrt{N\eta}\left(  \sqrt{1-\frac{1}{32}}-\frac{1}{8}\right)
>\sqrt{\frac{\eta N}{2}}.
\]
Hence, on $\overline{\Omega}$, $\hat{D}\geq N\eta$ and $\chi_{D}\geq
\sqrt{\left(  \eta N\right)  /2}$ for all $D$ such that $D\geq N\eta$ and
therefore $\chi_{\hat{D}}\geq\sqrt{\left(  \eta N\right)  /2}$. Finally,
\[
\mathbb{E}\left[  \left\Vert \hat{f}_{\hat{D}}\right\Vert ^{2}\right]
=\varepsilon^{2}\mathbb{E}\left[  \chi_{\hat{D}}^{2}\right]  \geq
\varepsilon^{2}\frac{\eta}{2}N\mathbb{P}\left\{  \chi_{\hat{D}}\geq\sqrt
{\frac{\eta N}{2}}\right\}  \geq\varepsilon^{2}\frac{\eta N}{2}\mathbb{P}%
\left(  \bar{\Omega}\right)  ,
\]
which, together with (\ref{eaux}) and (\ref{eminpro}) yields

\bigskip%
\[
\mathbb{E}\left[  \left\Vert \hat{f}_{\hat{D}}\right\Vert ^{2}\right]
\mathbb{\geq}\frac{(1-\delta)(1-\lambda)}{4}N\varepsilon^{2}%
\]
%

\endproof
\vspace{2mm}\newline

\noindent This result on the penalized selection procedure among the
collection $\left\{  \hat{f}_{D},1\leq D\leq N\right\}  $ tells us that if
$\lambda<1$, the penalized procedure with penalty $\operatorname*{pen}\left(
D\right)  =\lambda D\varepsilon^{2}$ has a tendency to select a high
dimensional model and that the risk of the selected estimator $\hat{f}%
_{\hat{D}}$ is order $N\varepsilon^{2}$ (up to some multiplicative constant
which tends to $0$ when $\lambda$ is close to $1$). This means that roughly
speaking the selected estimator behaves like the worst estimator of the
collection $\hat{f}_{N}$. Conversely, if $\lambda>1$, the right tail Gaussian
concentration inequality can be used to prove an oracle inequality (see
\cite{BM99g}). The proof being in some sense more standard we think that it is
useless to recall it here. This oracle inequality can be stated as follows.
For some constant $C\left(  \lambda\right)  $ depending only on $\lambda$,
whatever $f$%
\[
\mathbb{E}\left[  \left\Vert \hat{f}_{\hat{D}}-f\right\Vert ^{2}\right]  \leq
C\left(  \lambda\right)  \left(  \inf_{D\leq N}\mathbb{E}\left[  \left\Vert
\hat{f}_{D}-f\right\Vert ^{2}\right]  \right)  \text{.}%
\]
In the case where $f=0$, the best estimator among the collection $\left\{
\hat{f}_{D},D\leq N\right\}  $ is of course $\hat{f}_{1}$ and the oracle
inequality above tells us that when $\lambda>1$, the penalized estimator
$\hat{f}_{\hat{D}}$ has a quadratic risk which is of order of the quadratic
risk of $\hat{f}_{1}$ (i.e. $\varepsilon^{2}$), up to the multiplicative
constant $C\left(  \lambda\right)  $ (which of course tends to infinity when
$\lambda$ is close to $1$). Combining this with Proposition \ref{bou1} shows
that the value $\lambda=1$ is indeed critical since the penalized estimator
behaves like  $\hat{f}_{N}$ which is the worst estimator $\hat{f}_{N}$ when
$\lambda$ is below the critical value and like the best estimator $\hat{f}%
_{1}$ when $\lambda$ is above the critical value. Note that Proposition
\ref{bou1} is asymptotic in the sense that $N$ has to be large enough. The
practical consequence of this asymptotic nature of this critical phenomenon is
that if you run simulations, you will systematically see that some phase
transition occurs that some critical value $\hat{\lambda}_{\min}$ but  this
value is not necessarily close to $1$. Nevertheless the main point is that the
phase transition does occur and more importantly, the final estimator that you
intend to choose, which the penalized estimator with penalty
$\operatorname*{pen}\left(  D\right)  =2\hat{\lambda}_{\min}D\varepsilon^{2}$
has a good behavior (even better than the penalized estimator with penalty
corresponding to the Mallows $C_{p}$ proposal, i.e. with $\operatorname*{pen}%
\left(  D\right)  =2\hat{\lambda}_{\min}D\varepsilon^{2}$). 
\section{Bandwidth selection for kernel density estimation}
\label{sec:Kernel}
This section is devoted to another instance of estimator selection, namely kernel estimator selection. Kernel estimators are natural procedures when we consider the problem of density estimation. So, in the sequel, we consider $n$ i.i.d. observed random variables $X_1,\ldots,X_n$ belonging to $\R$ with unknown density $f$. Given $\H$ a set of positive bandwidths, for any $h\in\H$,  we denote $\hat f_h$ the classical kernel rule:
$$\hat f_h(x)=\frac{1}{n}\sum_{i=1}^nK_h(x-X_i),$$
with $K_h(\cdot)=\frac{1}{h}K\left(\frac{\cdot}{h}\right)$ and $K$ a fixed integrable bounded kernel. As in Section~\ref{sec:Penalized}, selecting a kernel estimate $\hat f_h$ is equivalent to selecting a parameter, namely here a bandwidth. A huge amount of literature has been devoted to this problem (see \cite{DL1,DJKP, GL14, GL10, mas-stflour, RRTM, Rig, silverman} and references therein) but we point out adaptive minimax approaches proposed by Lepski \cite{lep2} and its recent variation proposed by Goldenshluger and Lepski \cite{GL13}. 
As detailed in subsequent Section~\ref{sec:heuristic}, these methods are based on pair by pair comparisons of kernel rules. They enjoy optimal theoretical properties but suffer from tuning issues and high computation costs in particular in the multivariate setting (see \cite{BLR, DHRR}). In the sequel, we propose a new penalization methodology for bandwidth selection where concentration inequalities, in particular for U-statistics, play a key role (see Section \ref{sec:Proofs}). We first give heuristic arguments, then theoretical results are provided in the oracle setting. In particular, minimal and optimal penalties are derived under mild assumptions. Finally, we generalize our results to the multivariate setting and we derive rates of convergence of our procedure on Nikol'skii classes.

\medskip

In the sequel, we still denote $\|\cdot\|$ the classical $\L_2$-norm and $\langle\cdot,\cdot\rangle$ the associated scalar product. For any $p\in[1,+\infty]$, we denote $\|\cdot\|_p$ the classical $\L_p$-norm. We also set
$$f_h=\E[\hat f_h]=K_h\star f,$$ where $\star$ denotes the standard convolution product. 
\subsection{A new selection method: heuristics and definition}\label{sec:heuristic}
Recall that our aim is to find a data-driven method to select the best bandwidth, i.e. a bandwidth $h$ such that the risk 
$\E\|f-\hat f_h\|^2$
is minimum. Starting from the classical bias-variance decomposition
$$\E\|f-\hat f_h\|^2=\|f- f_h\|^2+\E\|f_h-\hat f_h\|^2=:b_h+v_h,$$
it is natural to consider a criterion of the form 
$${\rm Crit}(h)=\hat b_h + \hat v_h,$$
where $\hat b_h$ is an estimator of the bias $b_h$ and $\hat v_h$ an estimator of the variance $v_h$. Minimizing such a criterion is hopefully equivalent to minimizing the risk. Using that 
$v_h$ is (tightly) bounded by $\|K\|^2/(nh)$, 
we naturally set
$\hat v_h=\lambda \|K\|^2/(nh)$, with $\lambda$ some tuning parameter.
The difficulty lies in estimating the bias. Here we assume that $\hm$, the minimum of the bandwidths grid, is very small. In this case 
$f_{\hm}= K_{\hm}\star f$ is a good approximation of $f$, so that  $\|f_{\hm}-f_h\|^2$ is close to $b_h$. This is tempting to estimate this term by $\|\hat f_{\hm}- \hat f_h\|^2$ but doing this introduces a bias. Indeed, since 
$$\hat f_{\hm}- \hat f_h=(\hat f_{\hm}- f_{\hm}- \hat f_h+f_h)+(f_{\hm}- f_h)$$
we have the decomposition
\begin{equation}\label{decomp}
\E\|\hat f_{\hm}- \hat f_h\|^2=\|f_{\hm}- f_h\|^2+\E\|\hat f_{\hm}-\hat f_h-f_{\hm}+ f_h\|^2.
\end{equation}
But the centered variable $\hat f_{\hm}-\hat f_h-f_{\hm}+ f_h$ can be written
$$\hat f_{\hm}-\hat f_h-f_{\hm}+ f_h
=\frac{1}{n}\sum_{i=1}^n(K_{\hm}-K_h)(.-X_i)-\E((K_{\hm}-K_h)(.-X_i)).$$ So, the second term in the right hand side of \eqref{decomp} is of order
$ {n}^{-1}\int(K_{\hm}(x)-K_h(x))^2dx.$
Hence
$$\E\|\hat f_{\hm}- \hat f_h\|^2\approx\|f_{\hm}- f_h\|^2+\frac{\|K_{\hm}-K_h\|^2}n$$
and then 
$$b_h\approx \|f_{\hm}-f_h\|^2\approx \|\hat f_{\hm}- \hat f_h\|^2-\frac{\|K_{\hm}-K_h\|^2}n.$$
These heuristic arguments lead to the following criterion to be minimized:
\begin{equation}\label{critworse}
{\rm Crit}(h)=\|\hat f_{\hm}- \hat f_h\|^2-\frac{\|K_{\hm}-K_h\|^2}n + \lambda \frac{\|K_h\|^2}n.
\end{equation}
Thus, our method consists in comparing every estimator of our collection to the overfitting one, namely $\hat f_{\hm},$ before adding the penalty term
$$\pen_\lambda(h)=\frac{\lambda\|K_h\|^2-\|K_{\hm}-K_h\|^2}{n}.$$
We call it {\it Penalized Comparison to Overfitting}, abbreviated PCO in the sequel.

Let us now compare this new method to existing ones.
Actually, since $\hat f_{\hm}$ is close to $n^{-1}\sum_{i=1}^n \delta_{X_i}$, our criterion is very close to the classical penalized least squares contrast which is used in  regression context. More precisely, if $\hm\to 0$, 
$ \langle\hat f_{\hm}, \hat f_h\rangle \to n^{-1}\sum_{i=1}^n \hat f_h({X_i})$ and then, using \eqref{critworse}, 
\begin{eqnarray*}
{\rm Crit}(h)&\approx &\|\hat f_{\hm}\|^2+\| \hat f_h\|^2-\frac2n\sum_{i=1}^n \hat f_h({X_i})+\pen_\lambda(h)\\
&\approx &\|\hat f_{\hm}\|^2+L_n(\hat f_h)+\pen_\lambda(h)
\end{eqnarray*}
where $L_n(g)=\|g\|^2-2P_n(g)$. Since the first term does not play any role in the minimization, this is the criterion studied by \cite{lerasle-nelo}.

\noindent The classical Lespki's method \cite{lep1, lep2} consists in selecting a bandwidth $\hat h$ by using the rule
$$\hat h=\max\left\{h\in\H:\quad \|\hat f_{h'}-\hat f_h\|^2\leq V_n(h')\ \mbox{ for any } h'\in\H \mbox{ s.t. } h'\leq h\right\},$$
for some well-chosen bandwidth-dependent sequence $V_n(\cdot)$. Introduced in \cite{GL08}, the Goldenshluger-Lepski's methodology is a variation of the Lepski's procedure still based on pair-by-pair comparisons between estimators. More precisely, Goldenshluger and Lepski suggest to use the selection rule
$$\hat h=\argmin_{h\in\H}\left\{A(h)+V_2(h)\right\},$$
with $$A(h)=\sup_{h'\in \H}\left\{\|\hat f_{h'}- \hat f_{h\vee h'}\|^2-V_1(h')\right\}_+,$$
where $x_+$ denotes the positive part $\max(x,0)$, ${h \vee h'}=\max(h,h')$ and 
$V_1(\cdot)$ and $V_2(\cdot)$ are penalties to be suitably chosen (Goldenshluger and Lepski essentially consider $V_2=V_1$ or $V_2=2V_1$ in \cite{GL08,GL09,GL10,GL13}). 
The authors establish the minimax optimality of their method when $V_1$ and $V_2$ are large enough. However, observe that if $V_1=0$, then, under mild assumptions, 
$$A(h)=\sup_{h'\in \H}\|\hat f_{h'}- \hat f_{h\vee h'}\|^2\approx \|\hat f_{\hm}- \hat f_{h}\|^2$$
so that our method turns out to be exactly some degenerate case of the Goldenshluger-Lespki's method. We study its performances in the oracle setting.
\subsection{Oracle inequality}
As explained in Section~\ref{sec:heuristic}, we study the performances of $\hat f:=\hat f_{\hat h}$ with
$$\hat h=\argmin_{h\in\H}\left\{\|\hat f_h- \hat  f_{h_{\min}}\|^2+\pen_\lambda(h)\right\},$$
where $h_{min}=\min \H$ and  $\pen_\lambda(h)$ is the penalty suggested by heuristic arguments of Section~\ref{sec:heuristic}. We assume that $\max  \mathcal H$ is smaller than an absolute constant. We have the following result.

\begin{theorem}\label{io2}
Assume that $K$ is symmetric and $\int K(u)du=1$. Assume also that $\hm\geq \|K\|_\infty\|K\|_1/n$ and $\|f\|_\infty<\infty$. 
 Let $x\geq 1$ and $\varepsilon\in (0,1)$. If 
 $$\pen_\lambda(h)=\frac{\lambda\|K_h\|^2-\|K_{\hm}-K_h\|^2}{n}, \qquad \text{ with }\lambda>0,$$
then, with probability larger than $1-C_1|\H|e^{-x}$,
\begin{eqnarray*}
\| \hat  f_{\hat h}-f\|^2 &\leq &C_0(\varepsilon)\min_{h\in\H}\|\hat f_h-f\|^2\\&&+C_2(\varepsilon, \lambda)\|f_{\hm}-f\|^2
+C_3(\varepsilon,K,\lambda)\left(\frac{\|f\|_{\infty}x^2}{n}+\frac{x^3}{n^2\hm}\right),
\end{eqnarray*}
where $C_1$ is an absolute constant and 
$C_0(\varepsilon)=
\lambda+\varepsilon$ { if } $\lambda\geq 1$, 
$C_0(\varepsilon)=1/\lambda+\varepsilon$  { if } $0<\lambda< 1$. The constant $C_2(\varepsilon, \lambda)$ only depends on $\varepsilon$ and $\lambda$ and $C_3(\varepsilon,K,\lambda)$ only depends on $\varepsilon$, $K$ and $\lambda$.
\end{theorem}
The proof of Theorem~\ref{io2} can be found in Section~\ref{proof-io2} which shows that $C_2(\varepsilon, \lambda)$ and $C_3(\varepsilon,K,\lambda)$  blow up when $\varepsilon$ goes to 0.
We thus have established an oracle inequality, provided that the tuning parameter $\lambda$ is positive. The two last terms are remainder terms and are negligible under mild assumptions, as proved in the minimax setting in Corollary \ref{cor:rates}.


Note that when the tuning parameter $\lambda$ is fixed by taking $\lambda=1$, the leading constant $C_0(\varepsilon)$ is minimum and is equal to $1+\varepsilon$, so can be as close to 1 as desired, and in this case the penalty is 
$$\pen_{\lambda=1}(h)=\frac{\|K_h\|^2-\|K_{\hm}-K_h\|^2}{n}= \frac{2\langle K_h, K_{\hm} \rangle-\|K_{\hm}\|^2}{n}.$$
Since the last term does not depend on $h$, it can be omitted and we obtain an optimal penalty by taking
$$\pen_{opt}(h):=\frac{2\langle K_h, K_{\hm} \rangle}{n}.$$
The procedure associated with $\pen_{opt}$ (or associated with any penalty $\pen$ such that $\pen(h)-\pen_{opt}(h)$ does not depend on $h$)  is then optimal in the oracle approach. For density estimation, leading constants for oracle inequalities achieved by Goldenshluger and Lepski's procedure are $1+\sqrt{2}$ (if $\|K\|_1=1$) in \cite{lac-mass} and $1+3\|K\|_1$ in \cite{GL10}. To our kwnoledge, if the Goldenshluger and Lepski's procedure can achieve leading oracle constants of the form $1+\varepsilon$, for any $\varepsilon>0$, remains an open question.



\subsection{Minimal penalty}
In this section, we study the behavior of selected bandwidths when $\lambda$, the tuning parameter of our procedure, is  too small. We assume that the bias $\|f_{\hm}-f\|^2$ is negligible with respect to the integrated variance, which is equivalent to
\begin{equation}\label{hyphmin}
n\hm\|f_{\hm}-f\|^2=o(1).
\end{equation}
The following result is proved in Section~\ref{proof-penmin}.
\begin{theorem} \label{penmin} 
Assume that $K$ is symmetric and $\int K(u)du=1$. Assume also that $\|f\|_\infty<\infty$ and $\hm$ satisfies \eqref{hyphmin} and
$$\frac{\|K\|_\infty\|K\|_1}{n}\leq \hm  \leq \frac{(\log n)^\beta}{ n}$$ for some real $\beta$.
If $$\pen_\lambda(h)=\frac{\lambda\|K_h\|^2-\|K_{\hm}-K_h\|^2}{n}, \qquad \text{ with }\lambda<0,$$
then, for $n$ large enough, with probability larger than $1-C_1|\H|\exp(-(n/\log n)^{1/3})$, 
$$
\hat h\leq C(\lambda) \hm \leq C(\lambda) \frac{ (\log n)^\beta}{n}
$$
where $C_1$ is an absolute constant and $C(\lambda)=2.1-1/\lambda$.
\end{theorem}
Thus, if the penalty is too small (i.e. $\lambda<0$), our method selects $\hat h$ close to $\hm$ with high probability. This leads to an overfitting estimator. In the same spirit as in Section~\ref{sec:minimal}, we derive a minimal penalty, which is (up to additive constants)
$$\pen_{min}(h):=\frac{2\langle K_h, K_{\hm} \rangle-\|K_h\|^2}{n}.$$
The  interest of this result is not purely theoretical since it can be used to provide a data-driven choice of the tuning parameter $\lambda$ (see the discussion in Section~\ref{sec:minimal}).
%
%

\subsection{The multivariate case and adaptive minimax rates}
We now deal with the multivariate setting and we consider $n$ i.i.d. observed random vectors $X_1,\ldots,X_n$ belonging to $\R^d$ with unknown density $f$. Given $\H=\H_1\times \dots \H_d$  a set of multivariate positive bandwidths, for $h=(h_1,\ldots,h_d)\in\H$,  we now denote $\hat f_h$ the kernel rule:
$$\hat f_h(x)=\frac{1}{n}\sum_{i=1}^nK_h(x-X_i),$$
with for any $x=(x_1,\ldots,x_d)\in\R^d$,
$$K_h(x)=\frac{1}{h_1\dots h_d}K\left(\frac{x_1}{h_1}, \dots, \frac{x_d}{h_d}\right).$$
We consider $\hm=(h_{1,min}, \dots, h_{d,\min})$ with $h_{j,\min}=\min \H_j$ for any $j=1,\ldots,d$ and we study $\hat f:=\hat f_{\hat h}$ with
$$\hat h=\argmin_{h\in\H}\left\{\|\hat f_h- \hat  f_{h_{\min}}\|^2+\pen_\lambda(h)\right\}$$
and $\pen_\lambda(h)$ is the penalty similar to the one chosen in the univariate setting. We still assume that for $j=1,\ldots,d$, $\max\mathcal H_j$ is smaller than an absolute constant.
We obtain in the multivariate setting, results analog to Theorems~\ref{io2} and \ref{penmin}. 
\begin{theorem} \label{penminmultiv} 
Assume that $K$ is symmetric and $\int K(u)du=1$.  
Assume also that $\|f\|_\infty<\infty$ and $\hm$ satisfies
$$n\prod_{j=1}^dh_{j,\min}\times\|f_{\hm}-f\|^2=o(1)$$
 and
$$\frac{\|K\|_\infty\|K\|_1}{n}\leq \prod_{j=1}^d h_{j,\min}   \leq \frac{(\log n)^\beta}{n}$$ for some real $\beta$.
If $$\pen_\lambda(h)=\frac{\lambda\|K_h\|^2-\|K_{\hm}-K_h\|^2}{n}, \qquad \text{ with }\lambda<0,$$
then, for $n$ large enough, with probability larger than $1-C_1|\H|\exp(-(n/\log n)^{1/3})$,
$$
\prod_{j=1}^d \hat h_j\leq C(\lambda) \prod_{j=1}^d h_{j,\min} \leq C(\lambda) \frac{ (\log n)^\beta}{n}
$$
where $C_1$ is an absolute constant and $C(\lambda)=2.1-1/\lambda$.
\end{theorem}
As previously, if the penalty is too small, with high probability, each component of the selected bandwidth is (up to constants) the minimum one in each direction.
Let us state the analog of Theorem~\ref{io2}. 
\begin{theorem}\label{io2multiv}
Assume that $K$ is symmetric and $\int K(u)du=1$. Assume also that $\prod_{j=1}^d h_{j,\min}\geq \|K\|_\infty\|K\|_1/n$ and $\|f\|_\infty<\infty$. 
 Let $x\geq 1$ and $\varepsilon\in (0,1)$. If 
 $$\pen_\lambda(h)=\frac{\lambda\|K_h\|^2-\|K_{\hm}-K_h\|^2}{n}, \qquad \text{ with }\lambda>0,$$
then, with probability larger than $1-C_1|\H|e^{-x}$,
\begin{eqnarray*}
\| \hat  f_{\hat h}-f\|^2 &\leq &C_0(\varepsilon)\min_{h\in\H}\|\hat f_h-f\|^2\\
&&+C_2(\varepsilon, \lambda)\|f_{\hm}-f\|^2
+C_3(\varepsilon,K,\lambda)\left(\frac{\|f\|_{\infty}x^2}{n}+\frac{x^3}{n^2\prod_{j=1}^d h_{j,\min}}\right),
\end{eqnarray*}
where $C_1$ is an absolute constant and 
$C_0(\varepsilon)=
\lambda+\varepsilon$ { if } $\lambda\geq 1$, 
$C_0(\varepsilon)=1/\lambda+\varepsilon$  { if } $0<\lambda< 1$. The constant $C_2(\varepsilon, \lambda)$ only depends on $\varepsilon$ and $\lambda$ and $C_3(\varepsilon,K,\lambda)$ only depends on $\varepsilon$, $K$ and $\lambda$.
\end{theorem}
From the previous result, classical adaptive minimax anisotropic rates of convergence can be obtained. To be more specific, let us consider anisotropic Nikol'skii classes (see \cite{Nikol'skii} or \cite{GL14} for a clear exposition). For this purpose, let $(e_1,\ldots,e_d)$ denote the canonical basis of $\R^d$. For any function $g:\R^d\mapsto\R$ and any $u\in\R$, we define the first order difference operator with step size $u$ in the $j$-th direction by
$$\Delta_{u,j}g(x)=g(x+ue_j)-g(x),\quad j=1,\ldots,d.$$
By induction, the $k$-th order difference operator with step size $u$ in the $j$-th direction is defined as
$$\Delta_{u,j}^kg(x)=\Delta_{u,j}\Delta_{u,j}^{k-1}g(x)=\sum_{\ell=1}^k(-1)^{\ell+k}{k \choose \ell}\Delta_{u\ell,j}g(x).$$
We then set
\begin{definition}
For any given vectors ${\bf r}=(r_1,\ldots,r_d)$, $r_j\in[1,+\infty],$ $\boldsymbol{\beta}=(\beta_1,\ldots,\beta_d)$, $\beta_j>0,$ and ${\bf L}=(L_1,\ldots,L_d)$, $L_j>0,$ $j=1,\ldots,d$, we say that the function $g:\R^d\mapsto\R$ belongs to the anisotropic Nikol'skii class $\mathcal{N}_{{\bf r},d}(\boldsymbol{\beta},{\bf L})$ if
\begin{itemize}
\item[(i)] $\|g\|_{r_j}\leq L_j$ for all $j=1,\ldots,d$
\item[(ii)] for every $j=1,\ldots,d$, there exists a natural number $k_j>\beta_j$ such that
$$\|\Delta_{u,j}^{k_j}g\|_{r_j}\leq L_j|u_j|^{\beta_j},\quad \forall u\in\R^d,\quad \forall  j=1,\ldots,d.$$
\end{itemize}
\end{definition}
Note that the anisotropic Nikol'skii class is a specific class of the anisotropic Besov class \cite{KLP, Nikol'skii} :
$$\mathcal{N}_{{\bf r},d}(\boldsymbol{\beta},.)=\mathcal{B}^{\boldsymbol{\beta}}_{{\bf r},{\bf \infty}}(\cdot).$$
We consider the construction of the classical kernel $K$ proposed in Section~3.2 of \cite{GL14} such that assumptions of Theorem~\ref{io2multiv} are satisfied 
and 
$$\int K(u)u^kdu=0,\quad \forall |k|=1,\ldots,\ell-1,$$
where $k=(k_1,\ldots,k_d)$ is the multi-index, $k_i\geq 0,$ $
|k|=k_1+\cdots+k_d$ and $u^k=u_1^{k_1}\cdots u_d^{k_d}$ for $u=(u_1,\ldots,u_d)$. In this case, Lemma~3 of \cite{GL14} states that if $f\in\mathcal{N}_{{\bf 2},d}(\boldsymbol{\beta},{\bf L})$ with $\ell>\max_{j=1,\ldots,d}\beta_j$ then
$$f_h-f=\sum_{j=1}^dB_{j,h},$$
with the $B_j$'s satisfying
\begin{equation}\label{contbiais}
\|B_{j,h}\|\leq ML_jh_j^{\beta_j},
\end{equation}
for $M$ a positive constant depending on $K$ and $\boldsymbol{\beta}$. 
Finally, we consider $\H$ the following set of bandwidths: 
$$\H=\left\{h=(h_1,\ldots,h_d): \quad \frac{\|K\|_\infty\|K\|_1}{n} \leq \prod_{j=1}^d h_j, \ \
 h_j\leq 1\mbox{ and } h_j^{-1} \mbox{ is an integer } \forall\,j=1,\ldots,d 
\right\}.$$
Note that $|\H|\leq \tilde C_1(K,d)n^d$ where $\tilde C_1(K,d)$ only depends on $K$ and $d$.
By combining Theorem~\ref{io2multiv} and Proposition~\ref{mino}, standard computations lead to the following corollary (see Section~\ref{sec:cor}).
\begin{corollary}\label{cor:rates}
Consider the previous construction of the kernel $K$ (depending on a given integer $\ell$) and the previous specific set of bandwidths $\H$. For any $h\in\H$, we set $$\pen_\lambda(h)=\frac{\lambda\|K_h\|^2-\|K_{\hm}-K_h\|^2}{n}, \qquad \text{ with }\lambda>0.$$
For given ${\bf L}=(L_1,\ldots,L_d)$ and  $\boldsymbol{\beta}=(\beta_1,\ldots,\beta_d)$ such that $\ell>\max_{j=1,\ldots,d}\beta_j$ and for $B>0$, we denote $\tilde{\mathcal{N}}_{{\bf 2},d}(\boldsymbol{\beta},{\bf L},B)$ the set of densities bounded by $B$ and belonging to $\mathcal{N}_{{\bf 2},d}(\boldsymbol{\beta},{\bf L})$. Then, \begin{eqnarray*}
\sup_{f\in \tilde{\mathcal{N}}_{{\bf 2},d}(\boldsymbol{\beta},{\bf L},B)}\E\left[\|\hat  f_{\hat h}-f\|^2\right] &\leq &\tilde C_2(\boldsymbol{\beta},K, B, d,\lambda)\left(\prod_{j=1}^dL_j^{\frac{1}{\beta_j}}\right)^{\frac{2\bar\beta}{2\bar\beta+1}}n^{-\frac{2\bar\beta}{2\bar\beta+1}},
\end{eqnarray*}
where $\tilde C_2(\boldsymbol{\beta},K, B, d,\lambda)$ is a constant only depending on $\boldsymbol{\beta},$ $K,$ $B$, $d$ and $\lambda$ and
$$\frac{1}{\bar\beta}=\sum_{j=1}^d \frac{1}{\beta_j}.$$
\end{corollary}
Theorem~3 of \cite{GL14} states that up to the constant $\tilde C_2(\boldsymbol{\beta},K, B, d,\lambda)$, we cannot improve the rate achieved by our procedure. So, the latter achieves the adaptive minimax rate over the class $\tilde{\mathcal{N}}_{{\bf 2},d}(\boldsymbol{\beta},{\bf L},B)$. These minimax adaptive properties can be extended to the case of multivariate H\"{o}lder spaces if we restrict estimation to compactly supported densities (see arguments of \cite{BLR}) and, in the univariate case, to the class of Sobolev spaces.
\section{Conclusion}
To conclude, our method shares all the optimal theoretical properties of Lepski's and Goldenshluger-Lepski's methodologies in oracle and minimax approaches but its computational cost is much slower. Indeed, for any $h\in\H$, we only need to compare $\hat f_h$ with respect to $\hat f_{\hm}$, and not to all $\hat f_{h'}$ for $h'\in\H$ (or at least for a large subset of  $\H$). It is a crucial point in the case where $|\H|$ is of order $n^d$, as previously in the $d$-dimensional case. And last but not least, calibration issues of our methodology are also addressed with clear identifications of minimal and optimal penalties with simple relationships between them: for ${\rm Crit}(h)=\|\hat f_h- \hat  f_{h_{\min}}\|^2+\pen_\lambda(h),$
we have
$$\pen_{min}(h)=\frac{2\langle K_h, K_{\hm} \rangle-\|K_h\|^2}{n}\quad \mbox{and}\quad \pen_{opt}(h)=\frac{2\langle K_h, K_{\hm} \rangle}{n}.$$
Numerical performances of PCO will be studied in details in a further work. Extending our method to other loss functions than the quadratic loss if of course very tempting. Some preliminary numerical experiments with the Hellinger loss seem to indicate that it could work event though some theory remains to be built. In the same spirit one can wonder whether the PCO could be used for more general selection estimator selection issues than bandwidth selection for kernel density estimation.
\section{Proofs of Section \ref{sec:Kernel}}\label{sec:Proofs}
In this section, we denote for any function $g$, for any $t$,
$$\tilde g(t)=g(-t).$$
The notation $\square$ denotes an absolute constant that may change from line to line. The proofs use the following lower bound \eqref{minorisk} established in Proposition 4.1 of \cite{lerasle-nelo} combined with their Proposition~3.3.
\begin{proposition}\label{mino} Assume that $K$ is symmetric and $\int K(u)du=1$. Assume also that $\hm\geq \|K\|_\infty\|K\|_1/n$. Let $\Upsilon \geq (1+2\|f\|_\infty\|K\|_1^2)\|K\|_\infty/\|K\|^2$. 
For all $x\geq 1$ and for all $\eta\in (0,1)$, with probability larger than $1-\square|\H|e^{-x}$, for all $h\in \H$, each of the following inequalities hold
\begin{eqnarray}
\|f-\hat f_h\|^2\leq (1+\eta)\left(\|f-f_h\|^2+\frac{\|K_h\|^2}{n}\right)+\square\frac{\Upsilon x^2}{\eta^3 n},\label{majorisk}\\
\|f-f_h\|^2+\frac{\|K_h\|^2}{n}\leq
(1+\eta)\|f-\hat f_h\|^2+\square\frac{\Upsilon x^2}{\eta^3 n}.\label{minorisk}
\end{eqnarray}
\end{proposition}
We first  give a general result for the study of $\hat f:=\hat f_{\hat h}$.
\begin{theorem} \label{io} 
Assume that $K$ is symmetric and $\int K(u)du=1$. Assume also that $\hm\geq \|K\|_\infty\|K\|_1/n$ and $\|f\|_\infty<\infty$. Let $x\geq 1$ and $\theta\in (0,1)$. With probability larger than $1-C_1|\H|\exp(-x)$, for any $h\in\H$,
\begin{eqnarray*}
(1-\theta)\| \hat  f_{\hat h}-f\|^2 &\leq& (1+\theta)\|\hat f_h-f\|^2+\left(\pen_\lambda(h)-2\frac{\langle K_h, K_{\hm} \rangle}{n}\right)
-\left(\pen_\lambda(\hat h)-2\frac{\langle K_{\hat h}, K_{\hm} \rangle}{n}\right)
\\&&+\frac{C_2}{\theta}\|f_{\hm}-f\|^2+\frac{C(K)}{\theta }\left(\frac{\|f\|_{\infty}x^2}{n}+\frac{x^3}{n^2\hm}\right)
\end{eqnarray*}
where $C_1$ and $C_2$ are absolute constants and $C(K)$ only depends on $K$.
\end{theorem}
We prove in the sequel Theorem~\ref{io}, then Theorems \ref{io2} and \ref{penmin}. Theorems \ref{penminmultiv} and \ref{io2multiv} are proved similarly by replacing  $h$ by $\prod_{j=1}^d h_j$ and
$\hm$ by $\prod_{j=1}^d h_{j,\min}$, so we omit their proof. Section~\ref{sec:cor} gives the proof of Corollary~\ref{cor:rates}.
\subsection{Proof of Theorem~\ref{io}}
Let $\theta'\in(0,1)$ be fixed and chosen later. Using the definition of $\hat h$, we can write, for any $h\in \H$, 
\begin{eqnarray*}
\|\hat  f_{\hat h}-f\|^2+\pen_\lambda(\hat h)&=&\|\hat f_{\hat h}-\hat  f_{\hm}\|^2+\pen_\lambda(\hat h)+\|\hat  f_{\hm}-f\|^2+2\langle \hat f_{\hat h}-\hat  f_{\hm},\hat  f_{\hm}-f\rangle\\
&\leq&\|\hat f_{h}-\hat  f_{\hm}\|^2+\pen_\lambda(h)+\|\hat  f_{\hm}-f\|^2+2\langle \hat f_{\hat h}-\hat  f_{\hm},\hat  f_{\hm}-f\rangle\\
&\leq&\|\hat f_{h}-f\|^2+2\|f-\hat  f_{\hm}\|^2+
2\langle \hat f_h-f,f-\hat f_{\hm}\rangle+\pen_\lambda(h)\\
&&\hspace{1cm}+2\langle \hat f_{\hat h}-\hat  f_{\hm},\hat  f_{\hm}-f\rangle.
\end{eqnarray*}
Consequently,
\begin{equation}\label{idealpenal}
\| \hat  f_{\hat h}-f\|^2 \leq \|\hat f_h-f\|^2+\left(\pen_\lambda(h)-2\langle \hat f_h-f, \hat f_{\hm}-f\rangle\right)
-\left(\pen_\lambda(\hat h)-2\langle \hat f_{\hat h}-f, \hat f_{\hm}-f\rangle\right).
\end{equation}
Then, for a given $h$, we study the term 
$$2\langle \hat f_h-f, \hat f_{\hm}-f\rangle$$
that can be viewed as an ideal penalty.
Let us introduce the degenerate U-statistic
$$U(h,\hm)=\sum_{i\neq j}
\langle K_h(.-X_i)-f_h, K_{\hm}(.-X_j)-f_{\hm}\rangle 
$$
and the following centered variable
$$V(h,h')=<\hat f_h-f_h, f_{h'}-f>.$$
We first 
center the terms in the following way
\begin{eqnarray*}
\langle \hat f_h -f, \hat f_{\hm}-f\rangle
&=&\langle \hat f_h -f_h+f_h-f, \hat f_{\hm}-f_{\hm}+f_{\hm}-f\rangle\\
&=&\langle \hat f_h -f_h, \hat f_{\hm}-f_{\hm}\rangle
+\langle \hat f_h -f_h,  f_{\hm}-f \rangle\\
&&\hspace{2cm}
+\langle f_h -f, \hat f_{\hm}-f_{\hm}\rangle
+\langle f_h-f, f_{\hm}-f\rangle\\
&=&\langle \hat f_h -f_h, \hat f_{\hm}-f_{\hm}\rangle
+V(h,\hm)+V(\hm,h)+\langle f_h-f, f_{\hm}-f\rangle.
\end{eqnarray*}
Now,
\begin{eqnarray*}
\langle \hat f_h-f_h, \hat f_{\hm}-f_{\hm}\rangle
&=&\frac1{n^2}\sum_{i, j}\langle K_h(.-X_i)-f_h, K_{\hm}(.-X_j)-f_{\hm}\rangle \\
&=&\frac1{n^2}\sum_{i=1}^n\langle K_h(.-X_i)-f_h, K_{\hm}(.-X_i)-f_{\hm}\rangle+\frac{U(h,\hm)}{n^2}.
\end{eqnarray*}
Then
\begin{eqnarray*}
\langle \hat f_h-f_h, \hat f_{\hm}-f_{\hm}\rangle
&=&\frac{\langle K_h, K_{\hm}\rangle}{n}-\frac1{n}\langle \hat f_h,  f_{\hm}\rangle
-\frac1{n}\langle f_h, \hat f_{\hm}\rangle+\frac1{n}\langle f_h,  f_{\hm}\rangle
+\frac{U(h,\hm)}{n^2}.
\end{eqnarray*}
Finally, we have the following decomposition of $\langle \hat f_h-f, \hat f_{\hm}-f\rangle$:
\begin{eqnarray}
\langle \hat f_h-f, \hat f_{\hm}-f\rangle
&=&\frac{\langle K_h, K_{\hm} \rangle}{n} +\frac{U(h,\hm)}{n^2}
 \label{maint}\\
& &-\frac1n \langle \hat f_h, f_{\hm}\rangle-\frac1n \langle  f_h, \hat f_{\hm}\rangle+\frac1n \langle  f_h, f_{\hm}\rangle\label{negt}\\
& &+V(h,\hm)+ V(\hm,h)
+\langle f_h-f, f_{\hm}-f\rangle.\label{conct}
\end{eqnarray}
We first control the last term of the first line and we obtain the following lemma.
\begin{lemma}\label{lem:UStat}
With probability larger than $1-5.54|\mathcal H|\exp(-x)$, for any $h\in\mathcal H$,
$$\frac{|U(h,\hm)|}{n^2}\leq \theta' \frac{\|K\|^2}{nh}+\frac{\square \|K\|_1^2\|f\|_\infty x^2}{\theta' n}+\frac{\square\|K\|_{\infty}\|K\|_1 x^3}{\theta' n^2\hm}$$
\end{lemma}
\begin{proof}
We have:
\begin{eqnarray*}
U(h,\hm)&=&\sum_{i\neq j}
\langle K_h(.-X_i)-f_h, K_{\hm}(.-X_j)-f_{\hm}\rangle\\ 
&=&\sum_{i=2}^n\sum_{j<i}G_{h,\hm}(X_i,X_j)+G_{\hm,h}(X_i,X_j),
\end{eqnarray*}
with
$$G_{h,h'}(s,t)=\langle K_h(.-s)-f_h, K_{h'}(.-t)-f_{h'}\rangle.$$
Therefore, we can apply Theorem 3.4 of Houdr\' e and Reynaud-Bouret (2003):
$$\P\left(|U(h,\hm)|\geq \square\left(C\sqrt{x}+Dx+Bx^{3/2}+Ax^2\right)\right)\leq 5.54\exp(-x),$$
with $A$, $B$, $C$ and $D$ defined subsequently.
Since $$\|f_{\hm}\|_{\infty}=\|K_{\hm}\star f\|_{\infty}\leq \|K_{\hm}\|_{\infty}=\frac{\|K\|_{\infty}}{\hm},$$ we have
\begin{eqnarray*}
A&:=&\|G_{h,\hm}+G_{\hm,h}\|_\infty\\
&\leq&\|G_{h,\hm}\|_\infty+\|G_{\hm,h}\|_\infty\\
&\leq&4\|K_{\hm}\|_\infty\times (\|K\|_1+\|K_h\star f\|_1)\\
&\leq&\frac{8\|K\|_\infty\|K\|_1}{\hm}
\end{eqnarray*}
and 
$$\frac{Ax^2}{n^2}\leq \frac{8x^2\|K\|_\infty\|K\|_1}{n^2\hm}.$$
We have
$$
B^2:=(n-1)\sup_t\E[(G_{h,\hm}(t,X_2)+G_{\hm,h}(t,X_2))^2].
$$
and for any $t$,
\begin{eqnarray*}
\E[G_{h,\hm}^2(t,X_2)]&=&\E\left[\left(\int(K_h(u-t)-f_h(u))(K_{\hm}(u-X_2)-\E[K_{\hm}(u-X_2)])du\right)^2\right]\\
&\leq&\E\left[\int(K_h(u-t)-f_h(u))^2du\times\int(K_{\hm}(u-X_2)-\E[K_{\hm}(u-X_2)])^2du\right]\\
&\leq&2\times\int\left(K_h^2(u-t)+(K_h\star f)^2(u)\right)du\times\int\E[K_{\hm}^2(u-X_2)]du\\
&\leq&4\|K_h\|^2\times\|K_{\hm}\|^2.
\end{eqnarray*}
Therefore,
\begin{eqnarray*}
\frac{Bx^{3/2}}{n^2}&\leq&\frac{\theta'}{3} \frac{\|K\|^2}{nh}+\frac{6}{\theta'}\frac{\|K\|^2x^3}{ n^2\hm}.
\end{eqnarray*}
Now,
\begin{eqnarray*}
C^2&:=&\sum_{i=2}^n\sum_{j=1}^{i-1}\E[(G_{h,\hm}(X_i,X_j)+G_{\hm,h}(X_i,X_j))^2]\\
&\leq&\square\times n^2\E[G_{h,\hm}^2(X_1,X_2)]\\
&=&\square\times n^2\E\left[\left(\int(K_h(u-X_1)-f_h(u))(K_{\hm}(u-X_2)-f_{\hm}(u))du\right)^2\right]\\
&=&\square\times n^2\E\left[\left(\int K_h(u-X_1)K_{\hm}(u-X_2)du-\int K_h(u-X_1)(K_{\hm}\star f)(u)du\right.\right.\\
&&\hspace{3cm} -\left.\left.\int K_{\hm}(u-X_2)(K_h\star f)(u)du+\int (K_h\star f)(u)(K_{\hm}\star f)(u)du\right)^2\right].
\end{eqnarray*}
Since 
$$\|K_h\star f\|_{\infty}\leq \|K\|_1\|f\|_\infty,\quad \|K_h\star f\|_1\leq \|K\|_1,$$
we have:
\begin{eqnarray*}
C^2
&\leq&\square\times n^2\E\left[\left(\int K_h(u-X_1) K_{\hm}(u-X_2)du\right)^2\right]+\square\|K\|_1^4\|f\|_\infty^2\times n^2.
\end{eqnarray*}
So, we just have to deal with the first term.
\begin{eqnarray*}
\E\left[\left(\int K_h(u-X_1) K_{\hm}(u-X_2)du\right)^2\right]&=&\E\left[\left((\tilde K_h\star K_{\hm})(X_1-X_2)\right)^2\right]\\
&=&\iint (\tilde K_h\star K_{\hm})^2(u-v)f(u)f(v)dudv\\
&\leq&\|f\|_{\infty}\|\tilde K_h\star K_{\hm}\|^2\\
&\leq&\|K\|_1^2\|f\|_{\infty}\|K_h\|^2.
\end{eqnarray*}
Finally
$$C\leq\square\times n\|K\|_1\|f\|^{1/2}_{\infty}\|K_h\|+\square\|K\|_1^2\|f\|_\infty\times n.$$
So, since $x\geq 1$,
$$\frac{C\sqrt{x}}{n^2}\leq \frac{\theta'}{3} \frac{\|K\|^2}{nh}+\frac{\square\|K\|_1^2\|f\|_\infty\ x}{\theta' n}.$$
Now, let us consider 
$$\mathcal S:=\left\{a=(a_i)_{2\leq i\leq n},\ b=(b_i)_{1\leq i\leq n-1}:\quad \sum_{i=2}^n\E[a_i^2(X_i)]\leq 1,\ \sum_{i=1}^{n-1}\E[b_i^2(X_i)]\leq 1\right\}.$$
We have
\begin{eqnarray*}
D&:=&\sup_{(a,b)\in\mathcal S}\left\{\sum_{i=2}^n\sum_{j=1}^{i-1}\E[(G_{h,\hm}(X_i,X_j)+G_{\hm,h}(X_i,X_j))a_i(X_i)b_j(X_j)]\right\}.
\end{eqnarray*}
We have for $(a,b)\in\mathcal S$,
\begin{eqnarray*}
\sum_{i=2}^n\sum_{j=1}^{i-1}\E[G_{h,\hm}(X_i,X_j)a_i(X_i)b_j(X_j)]
&\leq&\sum_{i=2}^n\sum_{j=1}^{i-1}\E\int |K_h(u-X_i)-(K_h\star f)(u)||a_i(X_i)|\\
&&\hspace{2cm}\times|K_{\hm}(u-X_j)-(K_{\hm}\star f)(u)||b_j(X_j)|du\\
&\leq&\sum_{i=2}^n\sum_{j=1}^{n-1}\int \E\left[|K_h(u-X_i)-(K_h\star f)(u)||a_i(X_i)|\right]\\
&&\hspace{2cm}\times\E\left[|K_{\hm}(u-X_j)-(K_{\hm}\star f)(u)||b_j(X_j)|\right]du
\end{eqnarray*}
and for any $u$,
\begin{eqnarray*}
\sum_{i=2}^n\E |K_h(u-X_i)-(K_h\star f)(u)||a_i(X_i)|&\leq&\sqrt{n}\sqrt{\sum_{i=2}^n\E^2\left[K_h(u-X_i)-(K_h\star f)(u)||a_i(X_i)|\right]}\\
&\leq&\sqrt{n}\sqrt{\sum_{i=2}^n\E\left[|K_h(u-X_i)-(K_h\star f)(u)|^2\right]\E\left[a_i^2(X_i)\right]}\\
&\leq&\sqrt{n}\sqrt{\sum_{i=2}^n\E\left[K_h^2(u-X_i)\right]\E\left[a_i^2(X_i)\right]}\\
&\leq&\sqrt{n}\sqrt{\|f\|_\infty\|K_h\|^2\sum_{i=2}^n\E\left[a_i^2(X_i)\right]}\\
&\leq&\|K_h\|\sqrt{n\|f\|_\infty}.
\end{eqnarray*}
Straightforward computations lead to
\begin{eqnarray*}
\sum_{j=1}^{n-1}\int\E\left[|K_{\hm}(u-X_j)-(K_{\hm}\star f)(u)||b_j(X_j)|\right]du&\leq&2\|K\|_1\sum_{j=1}^{n-1}\E[|b_j(X_j)|]\\
&\leq&2\|K\|_1\sqrt{n}\sqrt{\sum_{j=1}^{n-1}\E^2[|b_j(X_j)|]}\\
&\leq&2\|K\|_1\sqrt{n}.
\end{eqnarray*}
Finally,
$$\sum_{i=2}^n\sum_{j=1}^{i-1}\E[G_{h,\hm}(X_i,X_j)a_i(X_i)b_j(X_j)]\leq 2n\|K\|_1\sqrt{\|f\|_\infty}\|K_h\|$$
and 
\begin{eqnarray*}
\frac{Dx}{n^2}&\leq&\frac{4\|K\|_1\sqrt{\|f\|_\infty}\|K_h\|x}{n}\\
&\leq&\frac{\theta'}{3} \frac{\|K\|^2}{nh}+\frac{12\|K\|_1^2\|f\|_\infty\ x^2}{\theta' n}.
\end{eqnarray*}
\end{proof}

From Lemma~\ref{lem:UStat}, we obtain the following result.
\begin{lemma}\label{lem: penid}
With probability larger than $1-9.54|\mathcal H|\exp(-x)$, for any $h\in\mathcal H$,
\begin{eqnarray*}
|\langle \hat f_h-f, \hat f_{\hm}-f\rangle-\frac{\langle K_h, K_{\hm} \rangle}{n}| &\leq& \theta'\|f_h-f\|^2+
\theta' \frac{\|K\|^2}{nh}\\
&&\hspace{-1cm}+\left(\frac{\theta'}{2}+\frac{1}{2\theta'}\right)\|f_{\hm}-f\|^2+\frac{C_1(K)}{\theta' }\left(\frac{\|f\|_{\infty}x^2}{n}+\frac{x^3}{n^2\hm}\right),
\end{eqnarray*}
where $C_1(K)$ is a constant only depending on $K$.
\end{lemma}
\begin{proof} We have to control \eqref{negt} and   \eqref{conct}. Let
$h$ and $h'$ be fixed. First, we have
\begin{eqnarray*}
\langle \hat f_h, f_{h'}\rangle&=&\frac{1}{n}\sum_{i=1}^n\int K_h(x-X_i)f_{h'}(x)dx\\
&=&\frac{1}{n}\sum_{i=1}^n(\tilde K_h\star f_{h'})(X_i).
\end{eqnarray*}
Therefore,
$$|\langle \hat f_h, f_{h'}\rangle|\leq \|\tilde K_h\star f_{h'}\|_\infty\leq \|f_{h'}\|_\infty\|K\|_1\leq \|f\|_\infty\|K\|_1^2,$$ which gives the following control of \eqref{negt}:
$$\left|-\frac1n \langle \hat f_h, f_{\hm}\rangle-\frac1n \langle  f_h, \hat f_{\hm}\rangle+\frac1n \langle  f_h, f_{\hm}\rangle\right|\leq \frac{3\|f\|_\infty \|K\|_1^2}n.$$
So it remains to bound the three terms of  \eqref{conct}. Note  that 
\begin{eqnarray*}
V(h,h')&=&<\hat f_h-f_h, f_{h'}-f>\\
&=&\frac{1}{n}\sum_{i=1}^n \left(g_{h,h'}(X_i)-\E[g_{h,h'}(X_i)]\right)
\end{eqnarray*}
 with 
$$g_{h,h'}(x)=\langle K_h(.-x), f_{h'}-f\rangle=(\tilde K_h\star(f_{h'}-f))(x).$$
Since
$$\|g_{h,h'}\|_\infty\leq \| K\|_1\|f_{h'}-f\|_\infty\leq \| K\|_1(1+\| K\|_1)\|f\|_\infty$$
and
$$\E[g_{h,h'}^2(X_1)]\leq \|f\|_\infty  \|\tilde K_h\star(f_{h'}-f)\|^2\leq \|f\|_\infty  \|K\|_1^2\|f_{h'}-f\|^2$$ then, with probability larger than $1-2e^{-x}$, Bernstein's inequality leads to
\begin{eqnarray*}
|V(h,h')|&\leq &\sqrt{\frac{2x}{n} \|f\|_\infty \|K\|_1^2\|f_{h'}-f\|^2}+\frac{x\|f\|_\infty\| K\|_1(1+\| K\|_1)}{3 n}\\
&\leq &\frac{\theta'}{2} \|f_{h'}-f\|^2+\frac{C_V(K)\|f\|_\infty x}{\theta' n},
\end{eqnarray*}
where $C_V(K)$ is a constant only depending on the kernel norm $\|K\|_1$. Previous inequalities are first applied with $h'=\hm$ and then we invert the roles of $h$ and $\hm$. To conclude the proof of the lemma, we use 
$$|\langle f_h-f, f_{\hm}-f\rangle|\leq \frac{\theta'}{2}\|f_h-f\|^2+\frac{1}{2\theta'}\|f_{\hm}-f\|^2.$$
\end{proof}
%
%
%
%

Now, Proposition~\ref{mino} gives, with probability larger than 
 $1-\square|\mathcal H|\exp(-x)$, for any $h\in\mathcal H$,
$$\|f_h-f\|^2+
\frac{\|K\|^2}{nh}\leq 2\|\hat f_h-f\|^2+C_2(K)\|f\|_\infty\frac{x^2}{n},$$
where $C_2(K)$ only depends on $K$. Hence, by applying Lemma~\ref{lem: penid}, with probability larger than 
$1-\square|\mathcal H|\exp(-x)$, for any $h\in\mathcal H$,
\begin{eqnarray*}
\left|\langle \hat f_h-f, \hat f_{\hm}-f\rangle-\frac{\langle K_h, K_{\hm} \rangle}{n}-\langle \hat f_{\hat h}-f, \hat f_{\hm}-f\rangle+\frac{\langle K_{\hat h}, K_{\hm} \rangle}{n}\right| \leq
 2\theta'\|\hat f_h-f\|^2\\ +2\theta'\|\hat f_{\hat h}-f\|^2
+\left(\theta'+\frac{1}{\theta'}\right)\|f_{\hm}-f\|^2+\frac{\tilde C(K)}{\theta' }\left(\frac{\|f\|_{\infty}x^2}{n}+\frac{x^3}{n^2\hm}\right),
\end{eqnarray*}
where $\tilde C(K)$ is a constant only depending on $K$.
It remains to use \eqref{idealpenal} and to choose $\theta'=\frac{\theta}{4}$ to conclude.
\subsection{Proof of Theorem~\ref{io2}}\label{proof-io2}
We set $\tau=\lambda-1$. Let $\varepsilon\in (0,1)$, and $\theta\in (0,1)$ depending on $\varepsilon$ to be specified later.
 Using Theorem~\ref{io} with the chosen penalty, we obtain, with probability larger than $1-\square|\H|\exp(-x)$, for any $h\in\H$,
\begin{multline}\label{resth1}
(1-\theta)\| \hat  f_{\hat h}-f\|^2 +\tau\frac{\|K_{\hat h}\|^2}{n} \leq (1+\theta)\|\hat f_h-f\|^2+\tau\frac{\|K_h\|^2}{n}
+\frac{C_2}{\theta}\|f_{\hm}-f\|^2\\+\frac{C(K)}{\theta }\left(\frac{\|f\|_{\infty}x^2}{n}+\frac{x^3}{n^2\hm}\right).
\end{multline}
We first consider the case where $\tau\geq 0$. Then  $\tau{\|K_{\hat h}\|^2}/{n} \geq 0$, and, using Proposition~\ref{mino}, with probability 
$1-\square|\H|\exp(-x)$
$$\tau\frac{\|K_h\|^2}{n}\leq
\tau(1+\theta)\|f-\hat f_h\|^2+\tau\frac{C'(K) \|f\|_\infty x^2}{\theta^3 n},
$$
where $C'(K)$ is a constant only depending on the kernel $K$. Hence, with probability $1-\square|\H|\exp(-x)$
\begin{eqnarray*}
(1-\theta)\| \hat  f_{\hat h}-f\|^2 \leq (1+\theta+\tau(1+\theta))\|\hat f_h-f\|^2
+\frac{C_2}{\theta}\|f_{\hm}-f\|^2\\+\frac{C(K)}{\theta }\left(\frac{\|f\|_{\infty}x^2}{n}+\frac{x^3}{n^2\hm}\right)
+\tau\frac{C'(K) \|f\|_\infty x^2}{\theta^3 n}.
\end{eqnarray*}
With $\theta=\varepsilon/(\varepsilon+2+2\tau)$, we obtain
\begin{eqnarray*}
\| \hat  f_{\hat h}-f\|^2 \leq (1+\tau+\varepsilon)\|\hat f_h-f\|^2
+\frac{C_2(\varepsilon+2+2\tau)^2}{(2+2\tau)\varepsilon}\|f_{\hm}-f\|^2\\+C''(K,\varepsilon,\tau)\left(\frac{\|f\|_{\infty}x^2}{n}+\frac{x^3}{n^2\hm}\right)
\end{eqnarray*}
where $C''(K,\varepsilon,\tau)$ is a constant only depending on $K,\varepsilon,\tau$.

Let us now study the case $-1<\tau\leq 0$. In this case
$\tau{\|K_{ h}\|^2}/{n} \leq 0$, and, using Proposition~\ref{mino}, with probability 
$1-\square|\H|\exp(-x)$
$$\tau\frac{\|K_{\hat h}\|^2}{n}\geq
\tau(1+\theta)\|f-\hat f_{\hat h}\|^2+\tau\frac{C'(K) \|f\|_\infty x^2}{\theta^3 n},$$ 
where $C'(K)$ is a constant only depending on the kernel $K$.
Hence, \eqref{resth1} becomes
\begin{eqnarray*}
(1-\theta+\tau(1+\theta))\| \hat  f_{\hat h}-f\|^2 \leq (1+\theta)\|\hat f_h-f\|^2
+\frac{C_2}{\theta}\|f_{\hm}-f\|^2\\+\frac{C(K)}{\theta }\left(\frac{\|f\|_{\infty}x^2}{n}+\frac{x^3}{n^2\hm}\right)
-\tau\frac{C'(K) \|f\|_\infty x^2}{\theta^3 n}
\end{eqnarray*}
With $\theta=(\varepsilon(\tau+1)^2)/(2+\varepsilon(1-\tau^2))<1$, we obtain, with probability $1-\square|\H|\exp(-x)$,
\begin{eqnarray*}
\| \hat  f_{\hat h}-f\|^2 \leq \left(\frac{1}{1+\tau}+\varepsilon\right)\|\hat f_h-f\|^2
+C''(\varepsilon,\tau)\|f_{\hm}-f\|^2\\+C'''(K,\varepsilon,\tau)\left(\frac{\|f\|_{\infty}x^2}{n}+\frac{x^3}{n^2\hm}\right).
\end{eqnarray*}
where $C''(\varepsilon,\tau)$ depends on $\varepsilon$ and $\tau$ and $C'''(K,\varepsilon,\tau)$ depends on $K\varepsilon,\tau$.
\subsection{Proof of Theorem~\ref{penmin}}\label{proof-penmin}
We still set $\tau=\lambda-1$. Set $\theta\in (0,1)$ such that $\theta<-(1+\tau)/5$.
Let us first write the result of Theorem~\ref{io} with the chosen penalty and $h=\hm$:
\begin{eqnarray*}
(1-\theta)\| \hat  f_{\hat h}-f\|^2 +\tau\frac{\|K_{\hat h}\|^2}{n} \leq (1+\theta)\|\hat f_{\hm}-f\|^2+\tau\frac{\|K_{\hm}\|^2}{n}
+\frac{C_2}{\theta}\|f_{\hm}-f\|^2\\+\frac{C(K)}{\theta }\left(\frac{\|f\|_{\infty}x^2}{n}+\frac{x^3}{n^2\hm}\right).
\end{eqnarray*}
Then inequality \eqref{majorisk} with $h=\hm$ gives, with probability $1-\square|\H|\exp(-x)$
$$
\|f-\hat f_{\hm}\|^2\leq (1+\theta)\left(\|f-f_{\hm}\|^2+\frac{\|K\|^2}{n\hm}\right)+\frac{C'(K) \|f\|_\infty x^2}{\theta^3 n}
$$
where $C'(K)$ is a constant only depending on the kernel $K$. That entails
\begin{eqnarray*}
(1-\theta)\| \hat  f_{\hat h}-f\|^2 +\tau\frac{\|K_{\hat h}\|^2}{n} \leq ((1+\theta)^2+C_2/\theta)\|f-f_{\hm}\|^2
+(\tau+(1+\theta)^2)\frac{\|K_{\hm}\|^2}{n}\\+\frac{C(K)}{\theta }\left(\frac{\|f\|_{\infty}x^2}{n}+\frac{x^3}{n^2\hm}\right)
+\square(1+\theta)\frac{C'(K)\|f\|_\infty x^2}{\theta^3 n}
\end{eqnarray*}
We denote $u_n:=\frac{\|f_{\hm}-f\|^2}{\|K\|^2/n\hm}$. The previous inequality can then be rewritten
\begin{eqnarray*}
(1-\theta)\| \hat  f_{\hat h}-f\|^2 +\tau\frac{\|K_{\hat h}\|^2}{n} \leq [((1+\theta)^2+C_2/\theta)u_n
+\tau+(1+\theta)^2]\frac{\|K_{\hm}\|^2}{n}\\+{C(K, \theta)}\left(\frac{\|f\|_{\infty}x^2}{n}+\frac{x^3}{n^2\hm}\right)
\end{eqnarray*}
where $C(K, \theta)$ is a constant only depending on $K$ and $\theta$. %
Next we apply  inequality \eqref{minorisk} to $h=\hat h$: with probability $1-\square|\H|\exp(-x)$
$$\frac{\|K\|^2}{n{\hat h}}\leq
2\|f-\hat f_{\hat h}\|^2+\frac{C'(K)\|f\|_{\infty} x^2}{ n},
$$
with $C'(K)$ only depending on $K$. We obtain that with probability $1-\square|\H|\exp(-x)$
\begin{eqnarray*}
[(1-\theta)/2+\tau]\frac{\|K\|^2}{n\hat h}\leq [((1+\theta)^2+C_2/\theta)u_n+\tau+(1+\theta)^2]\frac{\|K\|^2}{n\hm}\\
+{C'(K, \theta)}\left(\frac{\|f\|_{\infty}x^2}{n}+\frac{x^3}{n^2\hm}\right)
\end{eqnarray*}
where $C'(K, \theta)$ is a constant only depending on $K$ and $\theta$.
By assumption, $u_n=o(1)$, then there exists $N$
such that for $n\geq N$, $((1+\theta)^2+C_2/\theta)u_n\leq \theta$.

Let us now choose $x=(n/\log n)^{1/3}$. Using the bound on $\hm$, the remainder term is bounded in the following way: For $C'(K, \theta)$ only depending on $K$ and $\theta$,
\begin{eqnarray*}
\left(\frac{n\hm}{\|K\|^2}\right){C'(K, \theta)}\left(\frac{\|f\|_{\infty}x^2}{n}+\frac{x^3}{n^2\hm}\right)
&\leq &{C''(K, \theta,\|f\|_{\infty})}\left(x^2\hm+\frac{x^3}{n}\right)\\
& \leq & C''(K, \theta,\|f\|_{\infty})\left(\frac{(\log n)^{\beta-2/3}}{n^{1/3}}+\frac{1}{\log n}\right)=o(1).
\end{eqnarray*}
Then for $n$ large enough, this term is bounded by $\theta$.
Thus, for $n$ large enough, with probability $1-\square|\H|\exp(-(n/\log n)^{1/3})$
\begin{eqnarray*}
[(1-\theta)/2+\tau]\frac{\|K\|^2}{n\hat h}\leq [\theta+\tau+(1+\theta)^2+\theta]\frac{\|K\|^2}{n\hm}
\end{eqnarray*}
and then 
\begin{eqnarray*}
\frac{\hat h}{(1-\theta)/2+\tau}\geq \frac{\hm}{ \tau+1+5\theta}
\end{eqnarray*}
Note that $(1-\theta)/2+\tau<1+\tau<0$, and we have chosen $\theta$ such that  $\tau+1+5\theta<0$. Then
\begin{eqnarray*}
{\hat h}\leq \frac{\tau+(1-\theta)/2}{ \tau+1+5\theta} \hm.
\end{eqnarray*}
One can choose for example $\theta=-(\tau+1)/10$.
\subsection{Proof of Corollary~\ref{cor:rates}}\label{sec:cor}
Let $f\in \tilde{\mathcal{N}}_{{\bf 2},d}(\boldsymbol{\beta},{\bf L},B)$ and ${\mathcal E}$ the event corresponding to the intersection of events considered in Theorem~\ref{io2multiv} and Proposition~\ref{mino}. For any $A>0$, by taking $x$ proportional to $\log n$, $\P({\mathcal E})\geq 1-n^{A}$. Therefore, 
\begin{eqnarray*}
\E\left[\|\hat  f_{\hat h}-f\|^2\right] 
&\leq&\E\left[\|\hat  f_{\hat h}-f\|^21_{\mathcal E}\right]+\E\left[\|\hat  f_{\hat h}-f\|^21_{{\mathcal E}^c}\right] \\
&\leq &\tilde C_2(\boldsymbol{\beta},K, B, d,\lambda)\left(\prod_{j=1}^dL_j^{\frac{1}{\beta_j}}\right)^{\frac{2\bar\beta}{2\bar\beta+1}}n^{-\frac{2\bar\beta}{2\bar\beta+1}},
\end{eqnarray*}
where $\tilde C_2(\boldsymbol{\beta},K, B, d,\lambda)$ is a constant only depending on $\boldsymbol{\beta},$ $K,$ $B$, $d$ and $\lambda$.
We have used \eqref{contbiais} on ${\mathcal E}$ and 
for any $h\in\H$,
$$\|\hat f_h-f\|^2\leq 2\|f\|^2+\frac{2n\|K\|^2}{\|K\|_\infty\|K\|_1},$$
on ${\mathcal E}^c$.


\begin{thebibliography}{99}                                                                                               %

\bibitem {Akaike}\textsc{Akaike}, H. Information theory and an extension of
the maximum likelihood principle. In P.N. Petrov and F.~Csaki, editors,
\textit{Proceedings 2nd International Symposium on Information Theory}, pages
267--281. Akademia Kiado, Budapest, 1973.

\bibitem {arlot-bach}\textsc{Arlot, S. }and \textsc{Bach, F. }Data-driven
calibration of linear estimators with minimal penalties. In Y. Bengio, D.
Schuurmans, J. Lafferty, C. K. I. Williams, and A. Culotta, editors,
\textit{Advances in Neural Information Processing Systems} 22, pages 46--54, (2009).

\bibitem {arlot-mass}\textsc{Arlot, S., Massart, P. }Data-driven calibration
of penalties for least-squares regression. \textit{J. Mach. Learn. Res.,
}10:245--279 (electronic), (2009).

\bibitem {Baha}\textsc{Bahadur}, R.R. Examples of inconsistency of maximum
likelihood estimates. \textit{Sankhya Ser.A} \textbf{20}, 207-210 (1958).

\bibitem {bar-cov}\textsc{Barron}, A.R. and \textsc{Cover}, T.M. Minimum
complexity density estimation.{\Large {\small \ }}\textit{IEEE Transactions on
Information Theory} \textbf{37}, 1034-1054 (1991).

\bibitem {BBM}\textsc{Barron}, A.R., \textsc{Birg\'{e}}, L. and
\textsc{Massart}, P. Risk bounds for model selection via penalization.
\textit{Probab. Th. Rel. Fields}. \textbf{113}, 301-415 (1999).

\bibitem {bau}\textsc{Baudry, J-P, Maugis, C., Michel, B.} Slope heuristics:
overview and implementation. \textit{Statistics and Computing}, pages 1--16, (2011).

\bibitem{BLR}
\textsc{Bertin, K., Lacour, C. and Rivoirard, V.}
 Adaptive pointwise estimation of conditional density function. \textit{Annales de l'Institut Henri Poincar\'e Probabilit\'es et Statistiques}, {\bf 52}(2), 939--980 (2016).

\bibitem{bertin11:_adapt_dantiz} \textsc{Bertin,  K., Le Pennec, E. and Rivoirard, V. } Adaptive Dantzig density estimation. \textit{Annales de l'Institut Henri Poincar\'e Probabilit\'es et  Statistiques}, {\bf 47}(1), 43--74, (2011).

\bibitem {bic-tsy}\textsc{Bickel, P.J., Ritov, Y., }and\textsc{ Tsybakov, A.
B.} Simultaneous analysis of Lasso and Dantzig selector. \textit{Annals of
Statistics}. Volume 37, Number 4,1705-1732 (2009).

\bibitem {BM93}\textsc{Birg\'{e}}, L. and \textsc{Massart}, P.{\small \ }
Rates of convergence for minimum contrast estimators. \textit{Probab. Th.
Relat. Fields} \textbf{97}, 113-150 (1993).

\bibitem {BM97b}\textsc{Birg\'{e}}, L. and \textsc{Massart}, P. From model
selection to adaptive estimation. In \textit{Festschrift for Lucien Lecam:
Research Papers in Probability and Statistics} (D. Pollard, E. Torgersen and
G. Yang, eds.), 55-87 (1997) Springer-Verlag, New-York.

\bibitem {BM97a}\textsc{Birg\'{e}}, L. and \textsc{Massart}, P.. Minimum
contrast estimators on sieves: exponential bounds and rates of convergence.
\textit{Bernoulli}, \textbf{4 }(3)\textbf{, } 329-375 (1998).

\bibitem {BM99g}\textsc{Birg\'{e}}, L. and \textsc{Massart}, P. Gaussian model
selection. \textit{Journal of the European Mathematical Society}, n${{}%
^{\circ}}3$, 203-268 (2001).

\bibitem {minpen}\textsc{Birg\'{e}}, L. and \textsc{Massart}, P. Minimal
penalties for Gaussian model selection. Probab. Th. Rel. Fields 138, 33-73 (2007).

\bibitem {boucheron-2013}\textsc{Boucheron}, S., \textsc{Lugosi}, G. and
\textsc{Massart}, P. \textit{Concentration inequalities. }Oxford University
Press (2013).

\bibitem {dan-wood}\textsc{Daniel}, C.\ and \textsc{Wood}, F.S.
\textit{Fitting Equations to Data}. Wiley, New York (1971).

\bibitem{DL1} \textsc{Devroye, L. and Lugosi, G.}
\textit{Combinatorial methods in density estimation.}
Springer Series in Statistics. Springer-Verlag, New York (2001)


\bibitem {DoJo}\textsc{Donoho}, D.L. and \textsc{Johnstone}, I.M.{\small \ }%
Ideal spatial adaptation by wavelet shrinkage{\small . }\textit{Biometrika}
\textbf{81}, 425-455 (1994).

\bibitem {dojo2}\textsc{Donoho}, D.L. and \textsc{Johnstone}, I.M. Ideal
denoising in an orthonormal basis chosen from a library of bases. \textit{C.
R. Acad. Sc. Paris S\'{e}r. I Math.} \textbf{319}, 1317-1322 (1994).

\bibitem {djkpasymp}\textsc{Donoho}, D.L. and \textsc{Johnstone}, I.M.,
\textsc{Kerkyacharian}, G. and \textsc{Picard}, D. Wavelet
shrinkage:Asymptopia? \textit{J. R.\ Statist. Soc. B} \textbf{57}, 301-369 (1995).

\bibitem{DJKP} \textsc{Donoho}, D.L. and \textsc{Johnstone}, I.M.,
\textsc{Kerkyacharian}, G. and \textsc{Picard}, D.
Density estimation by wavelet thresholding. \textit{Ann. Statist.} 24, no. 2, 508Ð539, (1996).

\bibitem{DHRR}
\textsc{Doumic, M., Hoffmann, M., Reynaud-Bouret, P., and Rivoirard, V.} 
\newblock Nonparametric estimation of the division rate of a size-structured
  population.
\newblock {\em SIAM J. Numer. Anal.}, 50(2):925--950. (2012).

\bibitem {efro-pinsk}\textsc{Efroimovitch}, S.Yu.\ and \textsc{Pinsker, M.S.}
Learning algorithm for nonparametric filtering. \textit{Automat. Remote
Control} \textbf{11}, 1434-1440 (1984), translated from Avtomatika i
Telemekhanika \textbf{11}, 58-65.

\bibitem{GL08}
\textsc{Goldenshluger}, A. and \textsc{Lepski}, O. Universal pointwise selection rule in
  multivariate function estimation. \textit{Bernoulli}  \textbf{14}~(4), 1150--1190 (2008)
  
  \bibitem{GL09}
\textsc{Goldenshluger}, A. and \textsc{Lepski}, O. Structural adaptation via {$\Bbb
  L_p$}-norm oracle inequalities. \textit{Probab. Theory Related Fields} \textbf{143}~(1-2),
  41--71 (2009)


\bibitem{GL10}
\textsc{Goldenshluger}, A. and \textsc{Lepski}, O.  Bandwidth selection in kernel density
  estimation: oracle inequalities and adaptive minimax optimality. \textit{Ann.
  Statist.} \textbf{39}~(3), 1608--1632 (2011)


\bibitem{GL13}
\textsc{Goldenshluger}, A. and \textsc{Lepski}, O.  General selection rule from a family
  of linear estimators. \textit{Theory Probab. Appl.} \textbf{57}~(2), 209--226 (2013)

\bibitem{GL14}
\textsc{Goldenshluger}, A. and \textsc{Lepski}, O.  On adaptive minimax density estimation on $\R^d$. \textit{Theory Probab. Appl.} \textbf{159}~(3), 479--543 (2014)

\bibitem{KLP} 
\textsc{Kerkyacharian, G., Lepski, O. and Picard, D.}
Nonlinear estimation in anisotropic multiindex denoising. Sparse case. \textit{Theory Probab. Appl}. 52, no. 1, 58--77, (2008) 

\bibitem {lac-mass}\textsc{Lacour, C., P.}~\textsc{Massart, P.}~Minimal
penalty for Goldenschluger-Lepski method (2016)
$<$%
hal-01121989v2%
$>$%
. To appear in \textit{Stoch. Proc. Appl.}

\bibitem {lebar}\textsc{Lebarbier}, E. Detecting multiple change points in the
mean of Gaussian process by model selection. \textit{Signal Processing
}\textbf{85}, n$%
{{}^\circ}%
4$, 717-736 (2005).

\bibitem {lep1}\textsc{Lepskii}, O.V. On a problem of adaptive estimation in
Gaussian white noise. \textit{Theory Probab.\ Appl.}\ \textbf{36}, 454-466 (1990).

\bibitem {lep2}\textsc{Lepskii}, O.V. Asymptotically minimax adaptive
estimation I: Upper bounds. Optimally adaptive estimates. \textit{Theory
Probab.\ Appl.}\ \textbf{36}, 682-697 (1991).

\bibitem {lep3}\textsc{Lepskii}, O.V. Upper Functions for Positive Random
Functionals. II. Application to the Empirical Processes Theory, Part 1.
\textit{Mathematical Methods of Statistics}, vol. 22 p.83--99 (2013).

\bibitem {lerasle}\textsc{Lerasle, M.}. Optimal model selection in density
estimation. \textit{Ann. Inst. Henri Poincar\'{e} Probab. Stat.},
48(3):884--908, (2012).

\bibitem {lerasle-nelo}\textsc{Lerasle, M., Malter-magalahes, N. and
Reynaud-Bouret, P. }Optimal kernel selection for density estimation. To appear
in High dimensional probabilities VII: The Cargese Volume (2015).

\bibitem {lerasle-tak}\textsc{Lerasle, M.} and \textsc{Takahashi, D.Y. }Sharp
oracle inequalities and slope heuristic for specification probabilities
estimation in general random fields. \textit{Bernoulli}, 22, 1 (2016).

\bibitem {Mallows}\textsc{Mallows}, C.L. Some comments on $C_{p}$.
\textit{Technometrics} 15, 661-675 (1973).

\bibitem {mas-stflour}\textsc{Massart, P.}~ \textit{Concentration inequalities
and model selection }. Ecole d'\'{e}t\'{e} de Probabilit\'{e}s de Saint-Flour
2003. Lecture Notes in Mathematics 1896, Springer Berlin/Heidelberg (2007).

\bibitem {Nikol'skii}
\textsc{Nikol'skii, S. M.}
\textit{Priblizhenie funktsii mnogikh peremennykh i teoremy vlozheniya. (Russian) [Approximation of functions of several variables and imbedding theorems]}
Second edition, revised and supplemented. ``Nauka'', Moscow (1977).

\bibitem {pinsker}\textsc{Pinsker}, \textsc{M.S.} Optimal filtration of
square-integrable signals in Gaussian noise. \textit{Problems of Information
Transmission} \textbf{16}, 120-133 (1980).

\bibitem{RRTM} \textsc{Reynaud-Bouret, P.,  Rivoirard, V. and Tuleau-Malot, C.}
Adaptive density estimation: a curse of support? \textit{J. Statist. Plann. Inference} 141, no. 1, 115Ð139, (2011).

\bibitem{Rig} \textsc{Rigollet, P.}
Adaptive density estimation using the blockwise Stein method. 
textit{Bernoulli}, 12, no. 2, 351--370, (2006).

\bibitem {saumard1}\textsc{Saumard, A. }Optimal model selection in
heteroscedastic regression usingpiecewise polynomial functions.
\textit{Electron. J. Stat.}, 7:1184--1223, 2013.

\bibitem {schwartz}\textsc{Schwarz, G. }Estimating the dimension of a model.
\textit{Ann. of Statist. 6 (2}): 461-464 (1978).

\bibitem{silverman}
\textsc{Silverman, B.~W.} 
\textit{Density estimation for statistics and data analysis}.
Monographs on Statistics and Applied Probability. Chapman \& Hall,  London. (1986).

\bibitem {Tib}\textsc{Tibshirani, R. }Regression shrinkage and selection via
the lasso. J. Royal. Statist. Soc B., Vol. 58, No. 1, pages 267-288 (1996).
\end{thebibliography}
\end{document}